\documentclass[a4paper,11pt]{amsart}
\usepackage{amssymb, amsxtra}
\usepackage[all]{xy}
\usepackage{color, colortbl}
\usepackage{array}
\usepackage{braket}
\usepackage{scalefnt}
\usepackage{tikz}
\usetikzlibrary{calc,arrows,positioning, automata}

\tikzset{
    every node/.style={font=\sffamily\small},
    main node/.style={thick,circle,draw,font=\sffamily\Large}
}

\definecolor{lightgray}{gray}{0.91}

\usepackage[pagebackref=true, linktocpage=true]{hyperref}
\hypersetup{%
bookmarksnumbered=true,%
colorlinks=true,%
linkcolor=blue,%
citecolor=blue,%
urlcolor=blue,%
setpagesize=false,%
pdftitle={},%
pdfauthor={}}

\theoremstyle{plain}
\newtheorem{Thm}{Theorem}[section]
\newtheorem{Lem}[Thm]{Lemma}
\newtheorem{Cor}[Thm]{Corollary}
\newtheorem{Prop}[Thm]{Proposition}

\theoremstyle{definition}
\newtheorem{Def}[Thm]{Definition}
\newtheorem{Conj}[Thm]{Conjecture}
\newtheorem{Rem}[Thm]{Remark}
\newtheorem*{Ack}{Acknowledgments}
\newtheorem{Ex}[Thm]{Example}
\newtheorem{Setting}[Thm]{Setting}

\numberwithin{equation}{section}

\newcommand{\Proj}{\operatorname{Proj}}
\newcommand{\prt}{\partial}
\newcommand{\Sing}{\operatorname{Sing}}
\newcommand{\Spec}{\operatorname{Spec}}
\newcommand{\rank}{\operatorname{rank}}
\newcommand{\Cl}{\operatorname{Cl}}
\newcommand{\Pic}{\operatorname{Pic}}
\newcommand{\bNE}{\operatorname{\overline{NE}}}
\newcommand{\Bs}{\operatorname{Bs}}
\newcommand{\mult}{\operatorname{mult}}
\newcommand{\lcm}{\operatorname{lcm}}
\newcommand{\ord}{\operatorname{ord}}
\newcommand{\Supp}{\operatorname{Supp}}
\newcommand{\wt}{\operatorname{wt}}
\newcommand{\Qsm}{\operatorname{Qsm}}
\newcommand{\reg}{\operatorname{reg}}

\newcommand{\mbA}{\mathbb{A}}
\newcommand{\mbC}{\mathbb{C}}
\newcommand{\mbP}{\mathbb{P}}
\newcommand{\mbQ}{\mathbb{Q}}
\newcommand{\mbR}{\mathbb{R}}
\newcommand{\mbZ}{\mathbb{Z}}
\newcommand{\mcH}{\mathcal{H}}
\newcommand{\mcM}{\mathcal{M}}
\newcommand{\mcO}{\mathcal{O}}
\newcommand{\mcP}{\mathcal{P}}
\newcommand{\msp}{\mathsf{p}}
\newcommand{\msq}{\mathsf{q}}

\newcommand{\msi}{\mathsf{i}}

\usepackage{xcolor}

\newcommand{\ratmap}{\dashrightarrow}

\makeatletter
\def\imod#1{\allowbreak\mkern10mu({\operator@font mod}\,\,#1)}
\makeatother

\title[$2 n^2$-inequality and birational rigidity]{$2 n^2$-inequality for $cA_1$ points and applications to birational rigidity}

\author[I.~Krylov]{Igor~Krylov}
\address{Center for Geometry and Physics, Institute for Basic Science,79 Jigok-ro127beon-gil, Nam-gu, Pohang, Gyeongbuk, 37973, Korea}
\email{ikrylov@ibs.re.kr}

\author[T.~Okada]{Takuzo~Okada}
\address{Faculty of Mathematics, Kyushu University, Fukuoka 819-0385 Japan}
\email{tokada@math.kyushu-u.ac.jp}

\author[E.~Paemurru]{Erik~Paemurru}
\address{Department of Mathematics, University of Miami, Coral Gables, Florida 33146, USA.\\
Institute of Mathematics and Informatics, Bulgarian Academy of Sciences,\\
Acad.\ G.~Bonchev Str.\ bl.~8, 1113, Sofia, Bulgaria}
\curraddr{Mathematik und Informatik, Universität des Saarlandes, 66123 Saarbrücken, Germany}
\email{erik.paemurru@miami.edu}

\author[J.~Park]{Jihun~Park}
\address{Center for Geometry and Physics, Institute for Basic Science,79 Jigok-ro127beon-gil, Nam-gu, Pohang, Gyeongbuk, 37973, Korea}
\address{Department of Mathematics, POSTECH, 77 Cheongam-ro, Nam-gu, Pohang, Gyeongbuk, 37673, Korea}
\email{wlog@postech.ac.kr}

\date{}

\begin{document}

\begin{abstract}
The $4 n^2$-inequality for smooth points plays an important role in the proofs of birational (super)rigidity.
The main aim of this paper is to generalize such an inequality to terminal singular points of type $cA_1$, and obtain a $2 n^2$-inequality for $cA_1$ points.
As applications, we prove birational (super)rigidity of sextic double solids, many other prime Fano 3-fold weighted complete intersections, and del Pezzo fibrations of degree $1$ over $\mbP^1$ satisfying the $K^2$-condition, all of which have at most terminal $cA_1$ singularities and terminal quotient singularities.
These give first examples of birationally (super)rigid Fano 3-folds and del Pezzo fibrations admitting a $cA_1$ point which is not an ordinary double point.
\end{abstract}

\maketitle

\tableofcontents

\section{Introduction} \label{sec:intro}

Throughout the paper we work over the field $\mbC$ of complex numbers.

\subsection{$2 n^2$-inequality for $cA_1$ points}

Birational (super)rigidity of a Mori fiber space is roughly an essential uniqueness of the Mori fiber space structure in its birational equivalence class (see Definition~\ref{def:BR} for the precise definition).
Birational non-(super)rigidity of a Mori fiber space implies the existence of a non-biregular birational map to a Mori fiber space, and this further implies the existence of a mobile linear system that is highly singular.
In order to prove the birational (super)rigidity of a given Mori fiber space, one has to exclude the possibility of the existence of such a highly singular mobile linear system.
The $4n^2$-inequality, which is stated below, is quite useful in this context, and it is one of the most important ingredients in the proofs of birational (super)rigidity of Mori fiber spaces.

\begin{Thm}[{$4 n^2$-inequality, \cite[Theorem 2.1]{PukBook}}]
\label{thm:4nineq}
Let $\msp \in X$ be the germ of a smooth $3$-fold.
Let $\mcM$ be a mobile linear system on $X$ and let $n$ be a positive rational number.
If $\msp$ is a center of non-canonical singularities of the pair~$(X, \frac{1}{n} \mcM)$, then for general members $D_1, D_2 $ in $\mcM$ we have
\[
\mult_{\msp} D_1 \cdot D_2 > 4 n^2.
\]
\end{Thm}

To the best knowledge of the authors, there is no known example of a birationally (super)rigid $3$-dimensional Mori fiber space admitting a singularity other than quotient singularities and ordinary double points.
This is mainly because of the lack of a local inequality for singular points that is similar to $4 n^2$-inequality.

In this paper we consider a $3$-fold terminal singularity $\msp \in X$ of type $cA_k$ which is by definition an isolated hypersurface singularity whose general hyperplane section is the Du Val singularity of type $A_k$.
Our first goal is to obtain a similar inequality for singular points of type $cA_1$.

\begin{Thm}[$2 n^2$-inequality for $cA_1$ points]
Let $\msp \in X$ be the germ of a $cA_1$ singularity.
Let $\mcM$ be a mobile linear system on $X$ and let $n$ be a positive rational number.
If $\msp$ is a center of non-canonical singularities of the pair $(X, \frac{1}{n} \mcM)$, then for general members $D_1, D_2$ in $ \mcM$ we have
\[
\mult_{\msp} (D_1 \cdot D_2) > 2 n^2.
\]
\end{Thm}

As applications, we give first examples of birationally (super)rigid Fano $3$-folds and del Pezzo fibrations admitting $cA_1$ points.

\subsection{Applications to birational rigidity}

We explain applications of the $2 n^2$-inequality for $cA_1$ points to birational rigidity of some prime Fano $3$-folds and del Pezzo fibrations.

A {\it Fano $3$-fold} is a normal projective $\mbQ$-factorial variety of dimension $3$ with only terminal singularities whose anticanonical divisor is ample.
A Fano $3$-fold $X$ is {\it prime} if its class group $\Cl (X)$ is isomorphic to $\mbZ$ and is generated by $-K_X$.

\subsubsection{Sextic double solids}

A {\it sextic double solid} is a normal projective variety which is a double cover of $\mbP^3$ branched along a sextic surface.
Birational superrigidity of smooth sextic double solids is proved by Iskovskikh \cite{Isk80}, and later on, birational superrigidity of $\mathbb{Q}$-factorial sextic double solids with only ordinary double points is proved by Cheltsov and Park \cite{CP10}.

In \cite{Pae21}, birational geometry of sextic double solids with $cA_k$ points are investigated, and it is in particular proved that a general sextic double solid with a $cA_k$ singular point, where $k \ge 4$, is not birationally rigid.
Moreover the following expectation (which we pose as a conjecture) is made.

\begin{Conj}
A $\mbQ$-factorial sextic double solid with only terminal $cA_1$ and $cA_2$ singularities is birationally superrigid, and a $\mbQ$-factorial sextic double solid with only terminal $cA_1, cA_2$ and $cA_3$ singularities is birationally rigid.
\end{Conj}

We generalize the result of Cheltsov and Park \cite{CP10}, and prove the following.

\begin{Thm}[{$=$ Theorem~\ref{thm:SDS}}] \label{thm:sd}
Let $X$ be a $\mbQ$-factorial sextic double solid with only terminal singularities of type $cA_1$.
Then $X$ is birationally superrigid.
\end{Thm}

We also consider sextic double solids with a $cA_3$ point in Section~\ref{sec:SDSlinks} and construct a Sarkisov self-link under a generality assumption.

In Theorem~\ref{thm:sd} we are assuming that sextic double solids with $cA_1$ singularities are $\mbQ$-factorial, which is a crucial condition for them to be birationally superrigid.
However it is not a simple problem to determine whether a given singular variety is $\mbQ$-factorial or not.
We provide a criterion for $\mbQ$-factoriality of sextic double solids with $cA_1$ singularities in Section~\ref{sec:fac}.

\subsubsection{Prime Fano $3$-fold weighted complete intersections}

Quasismooth prime Fano $3$-fold WCIs are classified under some extra conditions and they consist of $95$ families of weighted hypersurfaces, $85$ families of WCIs of codimension $2$, and the family of complete intersections of $3$ quadrics in $\mbP^6$ (see \cite[16.6, 16.7]{IF} and \cite[Theorems 1.3, 6.1 and~7.4]{CCC}).
The study of birational (super)rigidity of these objects is almost completed under the assumption of quasi-smoothness:
\begin{itemize}
\item It is proved by \cite{CPR} and \cite{CP17} that every quasi-smooth prime Fano $3$-fold weighted hypersurface is birationally rigid.
\item It is proved by \cite{OkadaI} and \cite{AZ} that a quasi-smooth prime Fano $3$-fold WCI of codimension $2$ other than a complete intersection of a quadric and a cubic in $\mbP^5$ is birationally rigid if and only if it belongs to one of the specific $18$ families.
\item It is proved by \cite{IP96} that a general smooth complete intersection of a quadric and a cubic in $\mbP^5$ is birationally rigid.
\end{itemize}
It is known that a quasi-smooth weighted complete intersection has only cyclic quotient singularities (\cite[Theorem~3.1.6]{Dol82}).
Thus a quasi-smooth prime Fano $3$-fold WCI has only terminal cyclic quotient singularities.
We consider $78$ families of prime Fano $3$-fold weighted hypersurfaces and $18$ families of prime Fano $3$-fold WCIs of codimension $2$, and prove birational (super)rigidity of their special members admitting $cA_1$ points.

\begin{Thm}[{$=$ Theorem~\ref{thm:WCI}}] \label{mainthm:WCI}
Let $X$ be a prime Fano $3$-fold WCI which belongs to one of the families listed in Tables~\ref{table:Fanohyp} and~\ref{table:FanoWCI}.
Suppose that $X$ is quasi-smooth along the singular locus of the ambient weighted projective space, and $X$ has only $cA_1$ singularities besides terminal quotient singular points.
Then $X$ is birationally rigid.
\end{Thm}

\subsubsection{Del Pezzo fibrations of degree $1$}

Let $\pi \colon X \to \mbP^1$ be a del Pezzo fibration of degree $1$, that is, it is a Mori fiber space and its general fiber is a smooth del Pezzo surface of degree $1$.
We say that $X/\mbP^1$ satisfies the $K^2$-{\it condition} if the $1$-cycle $(-K_X)^2$ is not contained in the interior of the cone $\bNE (X)$ of effective curves on $X$, i.e.,
\begin{equation}\label{eq:K^2}
(-K_X)^2\not\in \mathrm{Int}\left(\bNE (X)\right).
\end{equation}
We refer readers to \cite{BCZ} for more details on $K^2$-condition and its related condition.

Pukhlikov \cite{PukdP} proved birational superrigidity of nonsingular del Pezzo fibrations satisfying the $K^2$-condition.
As an application of $2 n^2$-inequality for $cA_1$ points, we generalize the Pukhlikov's result and obtain the following.

\begin{Thm} \label{mainthm:dPfib}
Let $\pi \colon X \to \mbP^1$ be a del Pezzo fibration of degree $1$  with only $cA_1$ singularities.
If $X/\mbP^1$ satisfies the $K^2$-condition, then $X$ is birationally superrigid.
\end{Thm}

\begin{Ack}
The authors would like to express their gratitude to Byung Hee An for his assistance with Proposition~\ref{prop:Cl}.
The authors also would like to thank the referee for the careful reading and suggestions.
The first author was supported by KIAS Individual Grant n.\ MG069802 and IBS-R003-D1.
The second author is partially supported by JSPS KAKENHI Grant Numbers JP18K03216 and JP22H01118.
The third author is supported by the Simons Investigator Award HMS, National Science Fund of Bulgaria, National Scientific Program ``Excellent Research and People for the Development of European Science'' (VIHREN), Project No.~KP-06-DV-7.
The fourth author has been supported by IBS-R003-D1, Institute for Basic Science in Korea.
\end{Ack}

\section{Preliminaries}

\subsection{Birational (super)rigidity}

\begin{Def}
Let $\pi \colon X \to S$ be a morphism between normal projective varieties.
We say that $\pi \colon X \to S$ (or simply $X/S$ if $\pi$ is understood) is a {\it Mori fiber space} if
\begin{itemize}
\item $\dim S < \dim X$ and $\pi$ has connected fibers,
\item $X$ is $\mbQ$-factorial and has only terminal singularities,
\item $-K_X$ is $\pi$-ample and the relative Picard rank is $1$.
\end{itemize}
We call $X/S$ a {\it del Pezzo fibration} if $\dim X - \dim S = 2$.
\end{Def}

Note that Mori fiber spaces over a point are exactly Fano varieties of Picard rank $1$.

\begin{Def}
Let $\pi_X \colon X \to S$ and $\pi_Y \colon Y \to T$ be Mori fiber spaces.
A birational map $\chi \colon X \ratmap Y$ is called {\it square} if it fits into a commutative diagram
\[
\xymatrix{
X \ar[d]_{\pi_X} \ar@{-->}[r]^{\chi} & Y \ar[d]^{\pi_Y} \\
S \ar@{-->}[r]^{\xi} & T}
\]
where $\xi$ is birational and in addition the induced map on the generic fibers $\chi_{\eta} \colon X_{\eta} \ratmap Y_{\eta}$ is an isomorphism.
In this case we say that $X/S$ and $Y/T$ are {\it square birational}.
\end{Def}

\begin{Def} \label{def:BR}
We say that a Mori fiber space $\pi \colon X \to S$ is {\it birationally rigid} if, for any Mori fiber space $Y/T$, the variety $Y$ is birational to $X$ if and only if $X/S$ and $Y/T$ are square birational.

We say that $X/S$ is {\it birationally superrigid} if, for any Mori fiber space $Y/T$, any birational map $X \ratmap Y$ (if it exists) is square.
\end{Def}

Note that a Fano varietiy $X$ of Picard rank $1$ is birationally rigid if the only Mori fiber space it is birational to is itself, for example it is not birational to other Fano varieties.
Birationally rigid Fano variety is birationally superrigid if and only if any birational self-map is square and hence is an isomorphism.

\subsection{Maximal singularity}

For a normal variety $V$, a prime divisor $E$ on a normal variety $W$ admitting a projective birational morphism $\varphi \colon W \to V$ is called a {\it prime divisor over} $V$.
A prime divisor over $V$ is {\it exceptional} if its center on $V$ is a subvariety of codimension greater than $1$.

In the following, let $\pi \colon X \to S$ be a Mori fiber space.
For a mobile linear system $\mcM$ on $X$, the rational number $n \ge 0$ such that $\mcM \sim_{\mbQ} - n K_X + \pi^*A$ for some $\mbQ$-divisor $A$ on $S$ is called the {\it quasi-effective threshold} of $\mcM$.

\begin{Def}
A prime exceptional divisor $E$ over $X$ is a {\it maximal singularity} if there exists a mobile linear system $\mcM$ on $X$ such that
\[
\ord_E \varphi^*\mcM > n a_E (K_X),
\]
where $n > 0$ is the quasi-effective threshold of $\mcM$, $\varphi \colon Y \to X$ is a projective birational morphism such that $E \subset Y$ and $a_E (K_X)$ denotes the discrepancy of $K_X$ along $E$.
The center $\Gamma \subset X$ of a maximal singularity is called a {\it maximal center}.
An extremal divisorial contraction $\varphi \colon Y \to X$ is called a {\it maximal extraction} if its exceptional divisor is a maximal singularity.
\end{Def}

We have the following characterization of birational superrigidity for Fano $3$-folds of Picard rank $1$.

\begin{Thm}[{\cite[Theorem 1.26]{CS08} and \cite[(2.10) Proposition-definition]{Cor95}}]
Let $X$ be a Fano $3$-fold of Picard rank $1$.
Then the following are equivalent.
\begin{enumerate}
\item $X$ is birationally superrigid.
\item $X$ does not admit a maximal singularity.
\item $X$ does not admit a maximal extraction.
\end{enumerate}
\end{Thm}

Let $X$ be a Fano variety of Picard rank $1$, then we say that an a non-biregular birational map $\sigma \colon X \ratmap X'$ to a Fano variety of Picard number $1$ is an {\it elementary link of type II} if it sits in the commutative diagram
\[
\xymatrix{
Y \ar[d]_{\varphi} \ar@{-->}[r]^{\tau} & Y' \ar[d]^{\varphi'} \\
X \ar@{-->}[r]_{\sigma} & X'}
\]
where $\varphi$ and $\varphi'$ are extremal divisorial contractions and $\tau$ is a birational map which is an isomorphism in codimension $1$.
An elementary link of type II which is a birational self-map is called an {\it elementary self-link of type II}.

\begin{Lem}[{\cite[Lemmas 2.34 and 2.22]{OkadaII}}]  \label{lem:cribirrig}
Let $X$ be a Fano $3$-fold of Picard number $1$.
If for any maximal extraction $\varphi \colon Y \to X$ there exists a Sarkisov self-link of type II initiated by $\varphi$, then $X$ is birationally rigid.
\end{Lem}

\subsection{Exclusion methods}

We explain several methods which will be used to exclude maximal centers.

\subsubsection{Methods for curves}

\begin{Lem}[{\cite[Proof of Theorem~5.1.1 ]{CPR}} and {\cite[Lemma 2.9]{OkadaII}}] \label{lem:exclcurve1}
Let $X$ be a Fano $3$-fold of Picard rank $1$ and let $\Gamma \subset X$ be an irreducible and reduced curve.
If $(-K_X \cdot \Gamma) \ge (-K_X)^3$, then $\Gamma$ is not a maximal center.
\end{Lem}

The following is a simplified version of \cite[Lemma 2.11]{OkadaII}, which is enough for the purpose of this paper.

\begin{Lem}[{cf.\ \cite[Lemma 2.11]{OkadaII}}] \label{lem:exclcurve2}
Let $X$ be a Fano $3$-fold of Picard rank~$1$, and let $\Gamma \subset X$ be an irreducible and reduced curve.
Assume that there is a pencil~$\mcP$ of divisors on $X$ satisfying the following properties.
\begin{enumerate}
\item $\mcP \sim_{\mbQ} - m K_X$ for some rational number $m \ge 1$.
\item A general member of $\mcP$ is a normal surface.
\item For distinct general members $S, T \in \mcP$, we have $T|_S = \Gamma + \Delta$, where $\Delta$ is an irreducible and reduced curve such that $\Delta \ne \Gamma$, and we have $(\Gamma \cdot \Delta)_S \ge (-K_X \cdot \Delta)$.
\end{enumerate}
Then $\Gamma$ is not a maximal center.
\end{Lem}

When we apply Lemma~\ref{lem:exclcurve2}, we need to compute $(\Gamma \cdot \Delta)_S$, which will follow from the computation of $(\Gamma^2)_S$.
In this paper we need to consider the case where $\Gamma \cong \mbP^1$ and $S$  has Du Val singular points of type $A$ along $\Gamma$, and the computation of $(\Gamma^2)_S$ will be done by the following method.

\begin{Def}
Let $\msp \in S$ be the germ of a normal surface and $\Gamma$ an irreducible and reduced curve on $S$.
Let $\hat{S} \to S$ be the minimal resolution of $\msp \in S$ and denote by $E_1,\dots,E_m$ the prime exceptional divisors.
We define $G (S,\msp,\Gamma)$ to be the dual graph of $E_1,\dots,E_m$ together with the proper transform $\hat{\Gamma}$ of $\Gamma$ on $\hat{S}$: vertices of $G(S,\msp,\Gamma)$ corresponds to $E_1,\dots,E_m$ and $\hat{\Gamma}$, and two vertices corresponding to $E_i$ and $E_j$ (resp.\ $E_i$ and $\hat{\Gamma}$) are joined by $(E_i \cdot E_j)$-ple edge (resp.\ $(E_i \cdot \hat{\Gamma})$-ple edge).
We call $G (S,\msp,\Gamma)$ the {\it extended dual graph} of $(S,\msp,\Gamma)$.
\end{Def}

\begin{Def}
We say that $G (S,\msp,\Gamma)$ is {\it of type} $A_{n,k}$ if it is of the form
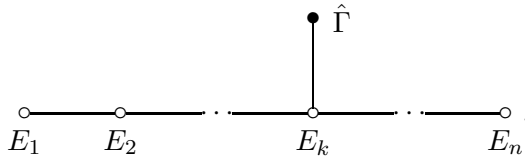
\[
\objectmargin={-0.5pt}
\xygraph{
\circ ([]!{+(0,-.3)} {E_1}) - [r]
\circ ([]!{+(0,-.3)} {E_2}) - [r] \cdots - [r]
\circ ([]!{+(0,-.3)} {E_k})(
-[u] \bullet ([]!{+(.3,0)} {\hat{\Gamma}}),
 - [r] \cdots - [r]
\circ ([]!{+(0,-.3)}{E_n}))}.
\]
Here, $\circ$ means that the corresponding exceptional divisor is a $(-2)$-curve.
In other words, $G (S,\msp,\Gamma)$ is of type $A_{n,k}$ if $(S,\msp)$ is of type $A_n$, $(\hat{\Gamma} \cdot E_i) = 0$ for $i \ne k$ and $(\hat{\Gamma} \cdot E_k) = 1$.
\end{Def}

\begin{Lem} \label{lem:selfintS}
Let $S$ be a normal projective surface, and let $\Gamma$ be a smooth rational curve on $S$.
Let $\msp_1, \dots, \msp_m$  be points on $\Gamma$. Suppose that $S$ is smooth along $\Gamma \setminus \{\msp_1, \dots, \msp_m\}$ and that each extended dual graph  $G (S, \msp_i, \Gamma)$ is of type $A_{n_i, k_i}$ for some positive integers $n_i, k_i$.
Then,
\[
(\Gamma^2) = - 2 - (K_S \cdot \Gamma) + \sum_{i=1}^m \frac{k_i (n_i - k_i + 1)}{n_i + 1}.
\]
\end{Lem}

\begin{proof}
This follows from \cite[Lemma 10.7]{OkadaIII}.
\end{proof}

\section{Local inequalities for $cA$ points}

Let $V$ be an $n$-dimensional variety.
For $i = 0, \dots, n$, we denote by $Z_i (V)$ the group of $i$-cycles on $V$, and by $A_i (V) = Z_i (V)/\sim_{\operatorname{rat}}$ the {\it Chow group of $i$-cycles} on $V$, where $\sim_{\operatorname{rat}}$ is the rational equivalence.
For Cartier divisors $D_1, \dots, D_k$ on $V$ that intersect properly, that is, the irreducible components of~$\Supp (D_1) \cap \cdots \cap\Supp (D_k)$ all have codimension equal to $k$ in $V$, there is a uniquely defined intersection cycle
\[
D_1 \cdot D_2 \cdots D_k \in Z_{n-k} (\Supp (D_1) \cap \cdots \cap \Supp (D_k)).
\]
See \cite[Definition 2.4.2]{Ful98}.

\subsection{Blow-ups and degrees of cycles}

Let $X$ be a normal $3$-fold, and let $B \subset X$ be an irreducible subvariety of codimension at least $2$.
Let $\varphi \colon \tilde{X} \to X$ be the blow-up of $X$ along $B$.
For a cycle $\Gamma \in Z_i(X)$ on $X$ such that $\operatorname{Supp} \Gamma  \not\subset B$, we denote by $\varphi_*^{-1} \Gamma$ the proper transform of $\Gamma$ on $\tilde{X}$.

\begin{Def}
Under the above setting, suppose that $B$ is not contained in the singular locus of $X$, and let $E \subset \tilde{X}$ be the unique exceptional divisor dominating~$B$.
Let $Z = \sum_i m_i Z_i$ be a $1$-cycle supported on the exceptional set $\varphi^{-1} (B)$, where $Z_i \subset \sigma^{-1} (B)$ is an irreducible curve.
We define the {\it degree} of $Z$ with respect to $E$ by
\[
\deg_E Z := \sum_i m_i \deg (Z_i \cap \varphi^{-1} (b)),
\]
where $\varphi^{-1} (b) \cong \mbP^{2 - \dim B}$ is the fiber over a general point $b \in B$.
\end{Def}

\begin{Lem}[{\cite[Lemma 2.2]{PukBook}}] \label{lem:degZsm}
Let $X$ be a normal $3$-fold, $B \subset X$ be an irreducible subvariety of codimension at least $2$ such that $B \not\subset \Sing X$, and let $D_1, D_2$ be Cartier divisors on $X$ which share no common component.
\begin{enumerate}
\item Assume that $B$ is a point.
Then, for the $1$-cycle
\[
Z := \varphi_*^{-1} D_1 \cdot \varphi_*^{-1} D_2 - \varphi_*^{-1} (D_1 \cdot D_2),
\]
which is supported on $\varphi^{-1} (B)$, we have
\[
\deg_E Z = \mult_B (D_1 \cdot D_2) - \mult_B D_1 \mult_B D_2.
\]
\item Assume that $B$ is a curve.
Then, for the decomposition
\[
\varphi_*^{-1} D_1 \cdot \varphi_*^{-1} D_2 = Z + \Gamma,
\]
where $\Supp Z \subset \varphi^{-1} (B)$ and $\Supp \varphi_* \Gamma$ does not contain $B$, we have
\[
D_1 \cdot D_2 = (\mult_B D_1 \mult_B D_2 + \deg_E Z) B + \varphi_* \Gamma.
\]
\end{enumerate}
\end{Lem}

\subsection{Blow-up of a $cA_1$ point}

By \cite{Reid1983}, a $3$-dimensional singularity is a terminal singularity if and only if it is the quotient of an isolated $cDV$ singularity by a suitable action of a cyclic group (of order $\ge 1$).
Note that a $cDV$ singularity is a hypersurface singularity whose general hyperplane section is a Du Val singularity.
Note also that a $cDV$ singular point is Gorenstein since it is a hypersurface singularity, and hence $3$-dimensional terminal Gorenstein singularities are exactly isolated $cDV$ singularities (see also \cite[Theorem~1.1]{Reid1983}).
We consider $cA_k$ singularities which are particular types of $cDV$ points.

\begin{Def}
Let $\msp \in X$ be the germ of a $3$-fold singularity.
We say that $\msp \in X$ is a \textit{singularity of type} $cA_k$ (or simply a $cA_k$ \textit{point}), if the germ is analytically equivalent to the germ of a hypersurface singularity whose general hyperplane section is the Du Val singular point of type $A_k$.
\end{Def}

An example of an isolated $cA_k$ is a point $o \in (f = 0) \subset \mbA^4$, where
\[
f = x y + z^{k+1} + w^M,
\]
for some $M \ge k+1$.
It is well known that any isolated $cA_1$-point is analytically equivalent to the above example with $k = 1$.

\begin{Lem}\label{analytic-equivalence}
Let $\msp \in X$ be an isolated $cA_1$-point, then it is analytically equivalent to $o \in (f = 0) \subset \mbA^4$, where
\[
f = x y + z^2 + w^M,
\]
for some $M \ge 2$.
\end{Lem}
\begin{proof}
By definition $\msp \in X$ is a hypersurface singularity, therefore without loss of generality we assume that $X$ is given by $g=0$ in $\mbA^4$ and $\msp = (0,0,0,0)$.
Since general hyperplane section of $X$ has the $A_1$-singularity the general form of $g$ is
\[
x^2 + y^2 + z^2 + t h(x,y,z,t) = 0.
\]
By the splitting lemma \cite[Theorem~I.2.47]{GLS07} we may change $g$ into
\[
x^2 + y^2 + z^2 + t^k = 0.
\]

\end{proof}

In the following we identify $\msp \in X$ with the above hypersurface germ, that is, we assume that $\msp \in X$ is the hypersurface germ defined by the vanishing of $f$ from Lemma \ref{analytic-equivalence}.
We set $W = \mbA^4$.
Let $\varphi_W \colon \tilde{W} \to W$ be the blowup of~$W$ at the origin $\msp$, $\tilde{X}$ the proper transform of $X$ in $\tilde{W}$, and $\varphi = \varphi_W|_{\tilde{X}} \colon \tilde{X} \to X$ the restriction.
We denote by $E_W \cong \mbP^3$ the exceptional divisor of $\varphi_W$ and we set $E = E_W|_{\tilde{X}}$.
The equation of $E$ in $E_W$ is given by
\[
xy + z^2 = 0,
\]
where $x,y,z,w$ are the coordinates inherited from $W$.
Thus we can see that $E$ is a quadratic cone.

For a cycle $\Gamma$ on $X$, we denote by $\mult_{\msp} \Gamma$ to be the multiplicity of $\Gamma$, viewed as a cycle on $W$, at the smooth point $\msp \in W$.

\begin{Def}
Under the above setting, let $Z$ be a $1$-cycle supported on $E_W$.
By the {\it degree} of $Z$ which is denoted by $\deg_{E_W} Z$, we mean the degree of $Z$ viewed as a $1$-cycle on $E_W \cong \mbP^3$.
By a slight abuse of notation, we use the notation $\deg_E Z = \deg_{E_W} Z$ when $Z$ is supported on $E$.
\end{Def}

\begin{Lem} \label{lem:degZcA1}
Let $\msp \in X$ be a germ of an isolated $cA_1$-singularity.
Let $\mcM$ be a mobile linear system of Cartier divisors on $X$, let $D_1, D_2$ be general members of $\mcM$, and denote $\nu(\mcM) = \ord_E \varphi^* \mcM$.
Let $\Gamma = D_1 \cdot D_2$ and denote $m = \mult_{\msp} \Gamma$.
Then, for the $1$-cycle
\[
Z := \varphi_*^{-1} D_1 \cdot \varphi_*^{-1} D_2 - \varphi_*^{-1} \Gamma,
\]
which is supported on $E$, we have
\[
\deg_E Z = m - 2 \nu(\mcM)^2.
\]
\end{Lem}

\begin{proof}
Let $i \colon X \to W$ and $\tilde{i} \colon \tilde{X} \to \tilde{W}$ be the embeddings agreeing with the maps $\varphi$ and $\varphi_W$.
Let $e \in A_1 (E_W)$ be the class of a line.
Then in $A_1 (\tilde{W})$, we have
\begin{equation} \label{eq:rateqZ1}
\varphi_W^* i_* \Gamma \equiv \varphi_*^{-1} i_* \Gamma + m e.
\end{equation}
Since $D_i$ is a Cartier divisor on $X$, it is defined by a single element in the residue ring $\mcO_{X, \msp} = \mcO_{W, \msp}/(x y + z^2 + w^M)$, hence we can take a Cartier divisor $G_i$ on $W$ such that $G_i|_X = D_i$.
It follows that
\[
\begin{split}
\varphi_W^* i_* \Gamma &= \varphi_W^* i_* (D_1 \cdot D_2) \\
&\equiv \varphi_W^* (G_1 \cdot G_2 \cdot X) \\
&\equiv \varphi^*_W G_1 \cdot \varphi_W^*G_2 \cdot \varphi_W^* X \\
&\equiv \varphi^*_W G_1 \cdot \varphi_W^*G_2 \cdot \tilde{X} \\
&\equiv \tilde{i}_* \big( \varphi^*D_1 \cdot \varphi^*D_2 \big),
\end{split}
\]
and thus
\begin{equation} \label{eq:rateqZ2}
\begin{split}
\tilde{i}_* (\varphi_*^{-1} D_1 \cdot \varphi_*^{-1} D_2)  &\equiv \tilde{i}_* \big( (\varphi^* D_1 - \nu(\mcM) E) \cdot (\varphi^* D_2 - \nu(\mcM) E) \big) \\
&\equiv \tilde{i}_* \big( \varphi^*D_1 \cdot \varphi^* D_2 \big) + \tilde{i}_* (\nu(\mcM)^2 E^2 ) \\
&\equiv \varphi_W^* i_* \Gamma - 2 \nu(\mcM)^2 e,
\end{split}
\end{equation}
where the last equivalence follows from \eqref{eq:rateqZ1}.
By \eqref{eq:rateqZ2}, the $1$-cycle $Z$ is rationally equivalent to $(m-2 \nu(\mcM)^2) e$ and the claim follows.
\end{proof}

\subsection{Tower of blow-ups associated to a divisorial valuation} \label{sec:tower}

We set up notations for proving $2 n^2$-type inequalities.

Let $\msp \in X$ be the germ of a terminal singular point of type $cA_1$.
Let $\nu$ be a divisorial valuation of $\mbC (X)$ realized by a prime divisor $E_{\infty}$ over $\msp \in X$.
Consider the tower of blow-ups realizing $\nu$:
\begin{equation} \label{eq:tower}
X_N \xrightarrow{\varphi_N} \cdots \xrightarrow{\varphi_2} X_1 \xrightarrow{\varphi_1} X_0 = X,
\end{equation}
where $\varphi_i \colon X_i \to X_{i-1}$ is the blow-up of $X_{i-1}$ along the center $B_{i-1}\subset E_{i-1}$ of $\nu$ on $X_{i-1}$, $E_i \subset X_i$ is the $\varphi_i$-exceptional divisor dominating $B_{i-1}$, and $E_N$ realizes the valuation $\nu$, i.e.\ $\nu = \nu_{E_N}$.
For $j > i$, the composite $X_j \to X_{j-1} \to \cdots \to X_i$ is denoted by $\varphi_{j, i} = \varphi_{i+1} \circ \cdots \varphi_j$.

For a cycle $\Gamma$ on $X_i$ and $j > i$, we denote by $\Gamma^{(j)}$ its proper transform on $X_j$ via $\varphi_{j, i} \colon X_j \to X_i$ if no component of $\Gamma$ is supported on $B_i$.
More generally for any object $(\cdot)$ (cycle, divisor, linear system) on $X_i$ and any $j > i$ we denote by $(\cdot)^{(j)}$ the proper transform of $(\cdot)$ on $X_j$.

We introduce an oriented graph structure as follows:
The vertices $E_j$ and $E_i$ are joined by an oriented edge, which is denoted by $j \to i$, if $j > i$ and $B_{j-1} \subset E_i^{(j-1)}$.
For $1 \le i < j \le N-1$, we define $P_{j,i}$ to be the number of paths from $E_j$ to $E_i$ in the oriented graph, we set $P_{i,i} = 1$ and we set $P_{j,i} = 0$ for $j < i$ we define $p_i = P_{N,i}$.
We define
\[
\begin{split}
K &= \max \{\, i \mid \text{$B_{i-1}$ is a $cA_1$ point} \,\}, \\
L &= \max \{\, i \mid \text{$B_{i-1}$ is a point} \,\}, \\
\Sigma_0 &= \sum_{i=1}^K p_i, \quad \Sigma_1 = \sum_{i=K+1}^L p_i, \quad \Sigma_2 = \sum_{i=L+1}^N p_i.
\end{split}
\]
Note that $B_0 = \msp$ and we have $1 \le K \le L \le N$.
The numbers $P_{i,j}, p_i, K, L$ are useful for computing discrepancies and tracking multiplicities for the tower of blow ups.

\begin{Lem}\label{lem:p_i-and-discrepancies}
The discrepancies of $E_i$ are given by the following
\[
a(X,E_i) = \sum_{j=1}^{\min \{i,K\}} P_{i,j} + 2\sum_{j=K+1}^{\min \{i,L\}} P_{i,j} + \sum_{j=L+1}^i P_{i,j},
\]
in particular
\[
a(X,E_N) = \sum_{j=1}^K p_j + 2\sum_{j=K+1}^L p_j + \sum_{j=L+1}^N p_j = \Sigma_0 + 2 \Sigma_1 + \Sigma_2.
\]
Let $D$ be a divisor on $X$ and set $K_D = \max \{i \mid B_{i-1}\subset D^{(i-1)} \}$, then
\[
\nu_{E_i}(D) = \sum_{j=1}^{\min \{K_D,i \}} P_{i,j} \nu_{E_j} (D^{j-1}),
\]
in particular
\[
\nu(D) = \sum_{j=1}^{K_D} p_j \nu_{E_j} (D^{j-1}).
\]
\end{Lem}
\begin{proof}
We prove the claims by induction.
The statements clearly hold for $i = 1$, now suppose the statements hold for all $k < i$.

First, we define values $a_i = a(E_i,X_{i-1})$ for all $i$, clearly
\[
a_i = \begin{cases}
1,\quad i\le K,\\
2, \quad K < i \le L,\\
1, \quad L \le N.
\end{cases}
\]

By assumption of induction we have
\[
K_{X^{(i-1)}} \sim \varphi_{i-1,0}^* (K_{X}) + \sum_{k=1}^{i-1} \Big( \sum_{j=1}^k a_j P_{k,j} \Big) E_k^{(i-1)}.
\]
Next we compute $K_{X^{(i)}}$:
\begin{align*}
K_{X^{(i)}} &\sim \varphi_i^* (K_{X^{(i-1)}}) + a_i E_i \sim \varphi_{i,0}^* (K_{X}) + \sum_{k=1}^{i-1} \Big( \sum_{j=1}^k a_j P_{k,j} \Big) \varphi_i^* E_k^{(i-1)} + a_i E_i \\
& \sim \varphi_{i,0}^* (K_{X}) + \sum_{k=1}^{i-1} \Big( \sum_{j=1}^k a_j P_{k,j} \Big) E_k^{(i)} + \Bigg( a_i P_{i,i} + \sum_{k \mid i\to k} \Big( \sum_{j=1}^k a_j P_{k,j} \Big) \Bigg)E_i.
\end{align*}
It remains to compute the coefficient at $E_i$
\begin{align*}
a(E_i,X) &= a_i P_{i,i} + \sum_{k\mid i\to k} \sum_{j=1}^k a_j P_{k,j} = a_i P_{i,i} + \sum_{k\mid i\to k} \sum_{j=1}^{i-1} a_j P_{k,j} \\
&= a_i P_{i,i} + \sum_{j=1}^{i-1} a_j \sum_{k\mid i\to k} P_{k,j} = \sum_{j=1}^i a_j P_{i,j},
\end{align*}
where the second equality follows from $P_{k,j} = 0$ for $j > k$ and the last equality follows from
\[
P_{i,j} = \sum_{k\mid i\to k} P_{k,j}.
\]
This proves the first assertion and its special case.
The proof of the second assertion is analogous.
\end{proof}

\begin{Ex}
Suppose in \ref{eq:tower} we have $B_{i} \not\subset E_{k}^{(i)}$ for all $k < i$.
Then the oriented graph is a simple chain:
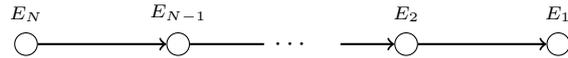
\begin{center}
\begin{tikzpicture}[node distance=1cm,auto]
\node[state, inner sep=3pt,minimum size=0pt](N) {};
 \node[state, inner sep=3pt,minimum size=0pt](N-1) at (2,0) {};
 \node[state, inner sep=3pt,minimum size=0pt] (2) at (5,0) {};
 \node[state, inner sep=3pt,minimum size=0pt] (1) at (7,0) {};

  \node (TN) at (0,.4)  {\tiny$E_{N}$};
   \node (TN1) at (2,.4)  {\tiny$E_{N-1}$};
    \node (T2) at (5,.4)  {\tiny$E_{2}$};
     \node (T1) at (7,.4)  {\tiny$E_{1}$};
   \node (O) at (3.5,0)  {$\cdots$};
    \node (O2) at (4,0)  {};

  \path[thick,->] (N) edge node[swap] {} (N-1);
   \path[thick,-] (N-1) edge node[swap] {} (O);
    \path[thick,->] (O2) edge node[swap] {} (2);
  \path[thick, ->] (2) edge node[swap] {} (1);
  \end{tikzpicture}
\end{center}

Clearly, in this case we have $p_1 = p_2 = \dots = p_N = 1$ and
\[
a(X,E_N) = K + 2 (L-K) + (N-L) = N + L - K.
\]
\end{Ex}

\begin{Ex}
Suppose $N = 4$, $B_2,B_3$ are points, $B_2\in E_1^{(2)} \cap E_2$, and $B_3 = E_1^{(3)}\cap E_2^{(3)} \cap E_3$, then the graph is

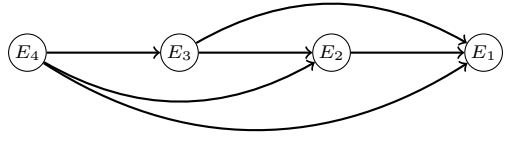
\begin{center}
\begin{tikzpicture}[node distance=2cm,auto]
\node[state, inner sep=1pt,minimum size=0pt] (4) {\tiny $E_4$};
 \node[state, inner sep=1pt,minimum size=0pt] (3) at (2,0) {\tiny$E_3$};
\node[state, inner sep=1pt,minimum size=0pt] (2) at (4,0) {\tiny$E_2$};
 \node[state, inner sep=1pt,minimum size=0pt]  (1) at (6,0){\tiny$E_1$};

 \path[thick,->] (4) edge node[swap] {} (3);
  \path[thick,->] (3) edge node[swap] {} (2);
   \path[thick, ->] (2) edge node[swap] {} (1);
 \path[thick, ->, bend right = 30] (4) edge node {} (2);
  \path[thick, ->, bend left= 30] (3) edge node {} (1);
   \path[thick, ->, bend right= 33] (4) edge node {} (1);
\end{tikzpicture}
\end{center}

Thus we can see that $p_1 = p_2 = 1$, $p_3 = 2$, and $p_4 = 4$.
\end{Ex}

Let $\mcM$ be a mobile linear system of Cartier divisors on $X = X_0$.
Let $Z_0 = D_1 \cdot D_2$ for general members $D_1, D_2 \in \mcM$.
Recall that we want to prove a lower bound for $\mult_{B_0} Z_0$.

\begin{Ex}
Suppose $N=1$, that is $\nu$ is realized by a single blow up.
Consider the intersection of the proper transforms $D_1^{(1)}, D_2^{(2)}$ on $X_1$.
By Lemma \ref{lem:degZcA1} it takes form
\[
D_1^{(1)} \cdot D_2^{(1)} = Z_0^{(1)} + Z_1,
\]
where
\[
\deg_{E_1} Z_1 = \mult_{B_0} Z_0 - 2 \nu(\mcM)^2 \ge 0.
\]
Note that $a(E_1,X,\frac{1}{n}\mcM) = (1 - \frac{\nu}{n})<0$ if and only if $\nu(\mcM) > n$.
Thus non-canonicity at divisor $E_1$ implies
\[
\mult_{B_0} Z_0 > 2n^2.
\]
\end{Ex}

\begin{Ex}
Suppose $N = 2$ and suppose that $B_1$ is a point such that $X_1$ is smooth at $B_1$.
As before, by Lemma \ref{lem:degZcA1} intersection on $X_1$ takes form
\[
D_1^{(1)} \cdot D_2^{(1)} = Z_0^{(1)} + Z_1,
\]
where
\[
\deg_{E_1} Z_1 = \mult_{B_0} Z_0 - 2 \nu_{E_1}(\mcM)^2.
\]
Using Lemma \ref{lem:degZsm} we compute similar intersection on $X_2$:
\[
D_1^{(2)} \cdot D_2^{(2)} = (D_1^{(1)} \cdot D_2^{(1)})^{(2)} + Z_2 = Z_0^{(2)} + Z_1^{(2)} + Z_2,
\]
where
\[
\deg_{E_2} Z_2 = \mult_{B_1} (Z_0^{(1)} + Z_1) - 4 \nu_{E_2}(\mcM^{(1)})^2.
\]
Since $\deg_{E_2} Z_2$ is nonnegative we have
\begin{align*}
0 \le \mult_{B_1} (Z_0^{(1)} + Z_1) - & 4  \nu_{E_2}(\mcM^{(1)})^2 \le \mult_{B_0} Z_0 + \deg_{E_1} Z_1 - 4  \nu_{E_2}(\mcM^{(1)})^2\\
&= 2 \mult_{B_0} Z_0 - 2 \nu_{E_1}(\mcM)^2 - 4  \nu_{E_2}(\mcM^{(1)})^2,
\end{align*}
or
\[
2 \mult_{B_0} Z_0 \ge 2 \nu_{E_1}(\mcM)^2 + 4  \nu_{E_2}(\mcM^{(1)})^2.
\]
It remains to use non-canonicity to find the lower bound on the RHS of this inequality.
\end{Ex}

In general we proceed similarly to the examples.
We define $\nu_i = \nu_{E_i} (\mcM^{(i-1)})$.
For general members $D_1, D_2 \in \mcM$, we define a sequence of $1$-cycles $Z_i$ on $X_i$ by
\begin{equation} \label{eq:defZi}
\begin{split}
D_1 \cdot D_2 &= Z_0, \\
D_1^{(1)} \cdot D_2^{(1)} &= Z_0^{(1)} + Z_1, \\
\cdots, \\
D_1^{(i)} \cdot D_2^{(i)} &= (D_1^{(i-1)} \cdot D_2^{(i-1)})^{(i)} + Z_i, \\
\cdots, \\
D_1^{(L)} \cdot D_2^{(L)} &= (D_1^{(L-1)} \cdot D_2^{(L-1)})^{(L)} + Z_L.
\end{split}
\end{equation}
Note that $\Supp Z_i \subset E_i$ and for any $i \le L$ we get
\[
D_1^{(i)} \cdot D_2^{(i)} = Z_0^{(i)} + Z_1^{(i)} + \cdots + Z_{i-1}^{(i)} + Z_i.
\]
For any $i, j$ with $L \ge j > i$, we set
\[
m_{i, j} = \mult_{B_{j-1}} (Z_i^{(j-1)}).
\]
Then there is a lower bound on multiplicities of $Z = Z_0$ at $B_i$ in terms of $\nu_i$.

\begin{Prop} \label{prop:cA1ineq}
Suppose $\msp \in X$ is a terminal singularity of type $cA_1$, and let $\nu$ be a divisorial valuation of $\mbC (X)$ centered at $\msp$.
Let $\mcM$ be a mobile linear system of Cartier divisors on $X$ and let $Z_0 = D_1 \cdot D_2$ be the intersection $1$-cycle of general members $D_1, D_2 $ in $ \mcM$.
Then, for $p_i, \nu_i$ and $m_{i, j}$ defined as above, the inequality holds:
\begin{equation} \label{eq:cA1ineq1}
2 \sum_{i=1}^K p_i \nu_i^2 + \sum_{i=K+1}^N p_i \nu_i^2 \le \sum_{i = 1}^L p_i m_{0,i}.
\end{equation}
\end{Prop}

\begin{proof}
Set $d_i = \deg_{E_i} Z_i$.
By Lemmas~\ref{lem:degZsm} and~\ref{lem:degZcA1}, we obtain a system of equalities corresponding to the sequence of cycles \eqref{eq:defZi}:
\begin{equation} \label{eq:syseqZi}
\begin{split}
2 \nu_1^2 + d_1 &= m_{0,1}, \\
2 \nu_2^2 + d_2 &= m_{0, 2} + m_{1, 2}, \\
\cdots, \\
2 \nu_K^2 + d_K &= m_{0, K} + \cdots + m_{K-1. K}, \\
\nu_{K+1}^2 + d_{K+1} &= m_{0, K+1} + \cdots + m_{K, K+1}, \\
\cdots, \\
\nu_L^2 + d_L &= m_{0, L} + \cdots + m_{L-1, L}.
\end{split}
\end{equation}
We multiply the $i$-th equality in \eqref{eq:syseqZi} by $p_i$ and sum them up:
\begin{equation} \label{eq:syseqZisum}
2 \sum_{i=1}^K p_i \nu_i^2 + \sum_{i=K+1}^L p_i \nu_i^2 + \sum_{i=1}^L p_i d_i
= \sum_{j=1}^L p_j m_{0, j} + \sum_{i = 1}^{L-1} \sum_{j = i+1}^L p_j m_{i, j}.
\end{equation}

For each $i$ with $0 \le i \le L-1$, we consider the expression
\[
\sum_{j = i+1}^L p_j m_{i, j}
\]
which appears in the RHS of the summation \eqref{eq:syseqZisum}.
We have $m_{i, j} \le d_i$ for any $i, j$ with $L \ge j > i \ge 1$ by \cite[Lemma 2.4]{PukBook}\footnote{This is valid for $j \le L$ by applying the lemma for the ambient fourfold.}.
Note that, for $i \ge 1$, we have $p_i \ge \sum_{j \to i} p_j$, and if $m_{i, j} \neq 0$, then $j \to i$.
It follows that, for $i = 1, \dots, L-1$, we have
\[
\sum_{j = i + 1}^L p_j m_{i, j} = \sum_{\substack{i + 1 \le j \le L \\ m_{i,j} \ne 0}} p_j m_{i,j} \le d_i \sum_{\substack{i + 1 \le j \le L \\ j \to i}} p_j \le p_i d_i.
\]
Thus, for $i = 1, \dots, L-1$, we can remove all $m_{i, *}$ and $p_i d_i$ from the summation \eqref{eq:syseqZisum} if we replace $=$ with $\le$, and obtain
\begin{equation} \label{eq:syseqZisumineq}
2 \sum_{i=1}^K p_i \nu_i^2 + \sum_{i = K+1}^L p_i \nu_i^2 + p_L d_L \le \sum_{i=1}^L p_i m_{0,i}.
\end{equation}

It remains to give a lower bound for $d_L$ in terms of $\nu_i$.
First, we decompose
\[
Z_L = \Gamma_L + \alpha_L B_L,
\]
where $\Supp \Gamma_L$ does not contain $B_L$ and $\alpha_L \ge 0$.
Thus we may decompose
\[
D_1^{(L+1)}\cdot D_2^{(L+1)} = Z_0^{(L+1)} + \dots + Z_{L-1}^{(L+1)} + \Gamma_L^{(L+1)} + Z_{L+1},
\]
where $\Supp Z_{L+1} \subset E_{L+1}$.
We apply Lemma \ref{lem:degZsm} (2) to $\varphi_{L+1}$:
\[
D_1^{(L)}\cdot D_2^{(L)} = (\varphi_{L+1})_* \big( Z_1^{(N)} + \dots + Z_{L-1}  + \Gamma_L^{(L+1)} \big) + \big( \nu_{L+1}^2 + \deg_{E_{L+1}} Z_{L+1} \big) B_L.
\]
Thus we have $\alpha_L = \nu_{L+1}^2 + \deg_{E_{L+1}} Z_{L+1}$ and
\[
d_L = \deg_{E_L} \big(\Gamma_L + (\nu_{L+1}^2 + \deg_{E_{L+1}} Z_{L+1}) B_L \big) \ge ( \nu_{L+1}^2 + \deg_{E_{L+1}} Z_{L+1} ) \deg_{E_L} B_L.
\]
Iterating the application of Lemma~\ref{lem:degZsm} (2) for $\varphi_i$ with $i = L + 2, \cdots, N$, we obtain
\[
d_L \ge \sum_{i = L + 1}^N \nu_i^2 \deg_{E_L} (\varphi_{i-1, L})_* B_{i-1} \ge \sum_{i= L + 1}^N \nu_i^2.
\]
Combining this inequality with \eqref{eq:syseqZisumineq}, we obtain the inequality \eqref{eq:cA1ineq1}.
\end{proof}

\begin{Cor} \label{cor:cA1ineq}
Let the notation and assumption be as in Proposition~\ref{prop:cA1ineq}.
Then the following inequality holds:
\[
\sum_{i = 1}^L p_i m_{0, i} \ge \frac{2 \nu (\mcM)^2}{\sum_{i=1}^K p_i + 2 \sum _{i=K+1}^N p_i}.
\]
\end{Cor}

\begin{proof}
The minimum of the LHS of \eqref{eq:cA1ineq1}, viewed as a quadratic form in $\nu_i$, under the constraint
\[
\sum_{i=1}^N p_i \nu_i = \nu (\mcM),
\]
is attained at
\[
2 \nu_1 = \cdots = 2 \nu_K = \nu_{K+1} = \cdots = \nu_N = \frac{2 \nu (\mcM)}{\sum_{i=1}^K p_i + 2 \sum_{i = K+1}^L p_i}.
\]
Calculating the minimum, we get the required inequality.
\end{proof}

\subsection{$2 n^2$-inequalities for $cA_1$ points}

We keep the notations as in Section~\ref{sec:tower}.
Recall that
\begin{equation} \label{eq:defSigmas}
\Sigma_0 = \sum_{i=1}^K p_i, \quad
\Sigma_1 = \sum_{i = K+1}^L p_i, \quad
\Sigma_2 = \sum_{i=L+1}^N p_i.
\end{equation}

\begin{Thm}[$2 n^2$-inequality for $cA_1$ points] \label{thm:ineqcA1}
Let $\msp \in X$ be a terminal singularity of type $cA_1$, $\mcM$ be a mobile linear system of Cartier divisors on $X$, and let~$n$ be a positive rational number.
If $\msp$ is a center of non-canonical singularities of the pair $(X, \frac{1}{n} \mcM)$, then for general members $D_1, D_2$ in $\mcM$ we have
\[
\mult_{\msp} (D_1 \cdot D_2) > 2 n^2.
\]
\end{Thm}

\begin{proof}
By the assumption, there exists a divisorial valuation $\nu$ of $\mbC (X)$ centered at $\msp$ such that $(X, \frac{1}{n} \mcM)$ is not canonical at $\nu$, that is, the inequality $\nu (\mcM) > n a(\nu,X)$ holds for $\nu$.
Thus by Lemma \ref{lem:p_i-and-discrepancies} we have
\begin{equation} \label{eq:2nineq1}
\nu (\mcM) > n (\Sigma_0 + 2 \Sigma_1 + \Sigma_2)
\end{equation}

Note that
\begin{equation} \label{eq:2nineq2}
\mult_{\msp} (D_1 \cdot D_2) = m_{0,1} \ge m_{0, i}
\end{equation}
for any $1 \le i \le L$.
By \eqref{eq:2nineq1}, \eqref{eq:2nineq2} and Corollary~\ref{cor:cA1ineq}, we obtain
\begin{equation} \label{eq:2nineqfin}
\mult_{\msp} (D_1 \cdot D_2) > 2 n^2 \frac{(\Sigma_0 + 2 \Sigma_1 + \Sigma_2)^2}{(\Sigma_0 + \Sigma_1)(\Sigma_0 + 2 \Sigma_1 + 2 \Sigma_2)}.
\end{equation}
It is then easy to get the required inequality.
\end{proof}

\begin{Thm}[$2 n^2$-type inequality for Kawakita blow-ups]
Let $\msp \in X$ be a terminal singularity of type $cA_1$, analytically equivalent to $o \in (f=0) \subset \mbA^4$, where
\[
f = xy + z^2 + w^M
\]
for some $M \geq 2$.
Let $n$ be a positive rational number, and $\mcM$ be a mobile linear system of Cartier divisors on $X$.
Let $E$ be the exceptional divisor of $(s, 2 t - s, t, 1)$-blow-up of $\msp \in X$, where $s, t$ are coprime integers such that $0 < s \le t \le M/2$.
If the pair $(X, \frac{1}{n} \mcM)$ is not canonical at $E$, then, for general members $D_1, D_2$ in $\mcM$, we have
\[
\mult_{\msp} (D_1 \cdot D_2) > \frac{2 t^2}{s (2 t - s)} n^2.
\]
\end{Thm}

\begin{proof}
Consider a fan corresponding to $\mbA^4$ with the following description
\begin{align*}
&v_x = (1,0,0,0),\\
&v_y = (0,1,0,0),\\
&v_z = (0,0,1,0),\\
&v_w = (0,0,0,1).
\end{align*}
Then the toric morphism corresponding to adding the ray $\rho = (s, 2t-s, t, 1)$ to the fan is the $(s, 2t-s, t, 1)$-weighted blow up of $\mbA^4$.
We can also add the ray $\rho$ by performing a sequence of regular blow ups at points, curves, and surfaces corresponding to adding the following rays to the fan:
\begin{align*}
u_1 &= (1,1,1,1),\\
u_2 &= (2,2,2,1),\\
&\dots\\
u_s &= (s,s,s,1),\\
u_{s+1} &= (s,s+1,s+1,1),\\
&\dots \\
u_{t} &= (s,t,t,1),\\
u_{t+1} &= (s,t+1,t,1),\\
&\dots\\
u_{2t-s} &= (s,2t-s,t,1).\\
\end{align*}

From this description we can see that the first $s$ morphisms are the blow ups at a points, first at $(x=y=z=w=0)$, then at $(x=y=z=u_{i}=0)$, the next $t-s$ morphisms are blow ups at curves $(y=z=u_i=0)$, and the last $t-s$ morphisms are blow ups at surface $(y=u_i=0)$.
The restriction of the morphisms to the proper transforms $X_i$ of $X$ gives us the $s$ blow ups at singular points of $X$ followed by $t-s$ blow ups at a curve and $t-s$ blow ups at a divisor as can be seen from the local equations of $X_i$
\begin{align*}
X_0:& \quad xy + z^2 + w^M = 0,\\
X_1:& \quad xy + z^2 + u_1^{M-2} w^M = 0,\\
 &\dots\\
X_s:& \quad xy + z^2 + u_s^{M-2s} w^M = 0,\\
X_{s+1}:& \quad xy + u_{s+1} z^2 + u_{s+1}^{M-2s-1} w^M = 0,\\
 & \dots\\
X_{t}:& \quad xy + u_{t}^{t-s} z^2 + u_{t}^{M-s-t} w^M= 0,\\
X_{t+1}:& \quad xy + u_{t+1}^{t-s-1} z^2 + u_{t+1}^{M-s-t-1} w^M =2,\\
 & \dots \\
X_{2t-s}:& \quad xy + z^2 + u_{2t-s}^{M-2t} w^M = 0.\\
\end{align*}

We can see from the toric description that for every blow up we have $B_i \subset E_i\setminus E_{i-1}^{(i)}$, so we can drop $u_j$ in the local equation of $X_i$ for $j<i$.
This fact also implies that the graph corresponding to the tower of blow ups is a simple chain and we have $p_i = 1$ for all $i$.
Thus we have $N = t$, $\Sigma_0 = K = L = s$, $\Sigma_1 = 0$, and $\Sigma_2 = t - s$.
The required inequality is obtained by substituting these values into the inequality \eqref{eq:2nineqfin}.
\end{proof}

\begin{Rem}
In Theorem~\ref{thm:ineqcA1}, compared to the $4n^2$-inequality (Theorem~\ref{thm:4nineq}), we keep the dimension three but we make the germ of a point worse by introducing singularities.
As a result, we weaken the original inequality.

On the other hand one can increase the dimension and see what kind of singularities still satisfy the $4n^2$-inequality, for example see \cite{Puk17}.
This approach has been applied to prove birational rigidity of many families of high dimensional varieties: hypersurfaces, cyclic covers, complete intersections (\cite{Puk19CI}, \cite{Puk19CO}, \cite{EP18}, \cite{EP19}).
\end{Rem}

\subsection{Corti inequality for $cA_1$ points}

We keep the notations as in Section~\ref{sec:tower} and let $\Sigma_0, \Sigma_1, \Sigma_2$ be as in \eqref{eq:defSigmas}.
The following is a version of Corti inequality \cite[Theorem 3.12]{Cor00} for $cA_1$ points.

\begin{Thm}[Corti inequality for $cA_1$ points] \label{thm:Cortiineq}
Let $\msp \in X$ be a terminal singularity of type $cA_1$, and let $\mcM$ be a mobile linear system of Cartier divisors on~$X$.
Let $F_1, \dots, F_l \subset X$ be irreducible surfaces containing $\msp$.
For the intersection $1$-cycle $Z = D_1 \cdot D_2$ of general members $D_1, D_2$ in $ \mcM$, we write
\[
Z = Z_h + \sum_{j=1}^l Z_j,
\]
where $\Supp Z_j \subset F_j$, and $Z_h$ intersects $\sum_{i=1}^l F_i$ properly.
Let $n $ and $\gamma_j $ be positive rational numbers such that $\msp$ is a center of non-canonical singularities of the pair
\begin{equation} \label{eq:Cortiineqpair}
\left( X, \frac{1}{n} \mcM - \sum_{j=1}^l \gamma_j F_j \right).
\end{equation}
Then there are rational numbers $0 < t_j \le 1$ such that
\[
\mult_{\msp} Z_h + \sum_{j=1}^l t_j \mult_{\msp} Z_j > 2 n^2 + 4 n^2 \sum_{j=1}^l \gamma_j t_j.
\]
\end{Thm}

\begin{proof}
We first note that by Lemma \ref{lem:p_i-and-discrepancies}
\[
a (E_{\infty}, X) = \Sigma_0 + 2 \Sigma_1 + \Sigma_2,
\]
where we recall that $E_{\infty} = E_N$ is the prime exceptional divisor realizing $\nu$.
By the assumption that the pair \eqref{eq:Cortiineqpair} is not canonical at $\nu$, we have
\begin{equation} \label{eq:Cortiineq1}
\nu (\mcM) > n \left( \Sigma_0 + 2 \Sigma_1 + \Sigma_2 + \sum_{j = 1}^k \gamma_j \nu (F_j) \right).
\end{equation}
For $j = 1, \dots, l$, we set
\[
K_j =
\begin{cases}
0, & \text{if $\msp \in F_j$}, \\
\max \left\{\, i \le L \mid B_{i-1} \notin F_j^{(i-1)} \, \right\}, & \text{otherwise},
\end{cases}
\]
and then set
\[
\Xi_j = \sum_{i=1}^{K_j} p_i.
\]
Then by Lemma \ref{lem:p_i-and-discrepancies} we have
\[
\nu (F_j) = \sum_{i = 1}^{K_j} p_i \mult_{B_{i-1}} F_j^{(i-1)} \ge \sum_{i = 1}^{K_j} p_i = \Xi_j.
\]

Observe that
\[
m_{i,0} = \mult_{B_{i-1}} Z_0^{(i-1)} = \mult_{B_{i-1}} Z_h^{(i-1)} + \sum_{j=1}^l \mult_{B_{i-1}} Z_j^{(i-1)},
\]
and that if $i > K_j$ then $\mult_{B_{i-1}} Z_j^{(i-1)} = 0$.
This decomposition implies that
\begin{equation} \label{eq:Cortiineq2}
\begin{split}
\sum_{i = 1}^L p_i m_{0, i} &= \sum_{i = 1}^L p_i \mult_{B_{i-1}} (Z_h^{(i-1)}) + \sum_{j = 1}^l \sum_{i=1}^{K_j} p_i \mult_{B_{i-1}} (Z_j^{(i-1)}) \\
&\le (\Sigma_0 + \Sigma_1) \mult_{\msp} Z_h + \sum_{j=1}^l \Xi_j \mult_{\msp} Z_j,
\end{split}
\end{equation}
since $\mult_{B_{i-1}} (Z_h^{(i-1)}) \le \mult_{\msp} Z_h$ and $\mult_{B_{i-1}} (Z_j^{(i-1)} )\le \mult_{\msp} Z_j$ for every $1 \le i \le L$.

By the inequality in Corollary~\ref{cor:cA1ineq} combined with two inequalities \eqref{eq:Cortiineq1}, \eqref{eq:Cortiineq2}, we get the inequality
\begin{equation} \label{eq:Cortiineq3}
\begin{split}
& (\Sigma_0 + \Sigma_1) \mult_{\msp} Z_h + \sum_{j = 1}^l \Xi_j \mult_{\msp} Z_j \\
&\ge \sum_{i=1}^L p_i m_{0,i} \\
&> 2 n^2 \Sigma_0 + 4 n^2 \Sigma_1 + 4 n^2 \sum_{j = 1}^l \gamma_j \nu (F_j) + 2 n^2 \frac{(\Sigma_2 - \sum_{j=1}^l \gamma_j \nu (F_j))^2}{\Sigma_0 + 2 \Sigma_1 + 2 \Sigma_2} \\
&\ge 2 n^2 (\Sigma_0 + \Sigma_1) + 4 n^2 \sum_{j=1}^l \gamma_j \Xi_j,
\end{split}
\end{equation}
where the last inequality follows from the inequality $\nu (F_j) \ge  \Xi_j$.

Dividing \eqref{eq:Cortiineq3} by $\Sigma_0 + \Sigma_1$ and setting
\[
t_j = \frac{\Xi_j}{\Sigma_0 + \Sigma_1},
\]
we get the required inequality.
\end{proof}

\section{Sextic double solids}

Recall that a sextic double solid is a normal projective $3$-fold $X$ which is a double cover of $\mbP^3$ branched along a sextic surface.
A sextic double solid can be expressed as a weighted hypersurface in $\mbP (1, 1, 1, 1, 3)$ with homogeneous coordinates $x, y, z, t, w$ of weights $1, 1, 1, 1, 3$, respectively, defined by an equation of the form
\begin{equation} \label{eq:SDSdefeq}
w^2 + g_6 (x, y, z, t) = 0,
\end{equation}
where $g_6 \in \mbC [x, y, z, t]$ is a homogeneous polynomial of degree $6$.

The aim of this section is to prove the following.

\begin{Thm} \label{thm:SDS}
Let $X$ be a factorial sextic double solid with at worst terminal singularities of type $cA_1$.
Then $X$ is birationally superrigid.
\end{Thm}

In the rest of this section, $X$ is a sextic double solid and we assume that $X$ is factorial and has only terminal Gorenstein singularities.
We do not impose any other condition on the singularities unless otherwise specified.
We denote by $\pi \colon X \to \mbP^3$ the double cover and by $B \subset \mbP^3$ the branch divisor which is a sextic surface.

\begin{Rem} \label{rem:sdfact}
Under the assumption that a sextic double solid $X$ has only Gorenstein terminal singularities, $\mbQ$-factoriality of $X$ is equivalent to factoriality of $X$, that is, $\Pic (X) = \Cl (X) = \mbZ [-K_X]$ (see \cite[Lemma 6.3]{Kawamata88}).
\end{Rem}

\subsection{Exclusion of smooth points}

The exclusion of smooth points as a maximal center is done as follows.

\begin{Prop}[{\cite[Lemma 3.1]{CP10}}] \label{prop:SDSsmpt}
No smooth point of  $X$ is a maximal center.
\end{Prop}

\subsection{Exclusion of curves}

\begin{Lem} \label{lem:SDmostcurves}
A curve on $X$ is not a maximal center except possibly for a curve  of degree $1$ which passes through a singular point of $X$ and whose image via $\pi$ is not contained in the sextic surface $B$.
\end{Lem}

\begin{proof}
Let $\Gamma$ be a curve on $X$.
Suppose that $\Gamma$ is a maximal center.
Then $(- K_X \cdot \Gamma) < (-K_X)^3 = 2$ by Lemma~\ref{lem:exclcurve1}.
It follows that $\Gamma$ is a curve of degree~$1$ on~$X$.
By the assumption, there exists a movable linear system $\mcM \sim - n K_X$ on~$X$ such that the pair $(X, \frac{1}{n} \mcM)$ is not canonical along $\Gamma$.
In particular we have $\mult_{\Gamma} \mcM > n$.

In the following, we merely repeat the arguments in the proof of \cite[Lemma 3.8]{CP10} for readers' convenience.
Suppose that $\pi (\Gamma)$ is contained in the sextic surface $B \subset \mbP^3$.
We choose and fix a smooth point $\msp \in \Gamma$ which is contained in the smooth locus of $X$ and which is not contained in any curve in $\Bs \mcM$ other than $\Gamma$.
Let $L \subset \mbP^3$ be a general line tangent to $B$ at $\pi (\msp)$ and let $\Delta$ be the inverse image of $L$ by $\pi$.
We see that $\Delta$ is a curve of degree $2$ which is singular at $\msp$.
Moreover, by our choice of $\msp$, $\Delta$ is not contained in the base locus of $\mcM$.
Then we obtain
\[
2 n = (\mcM \cdot \Delta) \ge \mult_{\msp} \mcM \mult_{\msp} \Delta \ge \mult_{\Gamma} \mcM \mult_{\msp} \Delta > 2 n.
\]
This is a contradiction.
\end{proof}

In the following, let $\Gamma$ be a curve of degree $1$ such that $\Gamma \cap \Sing X \ne \emptyset$ and $\pi (\Gamma) \not\subset B$.
The image $\pi (\Gamma) \subset \mbP^3$ is a line and we can write
\[
\Gamma = (\ell_1 = \ell_2 = w - g_3 = 0),
\]
where $\ell_1, \ell_2 \in \mbC [x, y, z, t]$ are linear forms and $g_3 \in \mbC [x, y, z, t]$ is a cubic form.
We define $\mcP$ to be the linear subsystem of $\left| -K_X \right|$ consisting of members containing $\Gamma$, which is the pencil generated by $\ell_1$ and $\ell_2$, and let $S, T$ be distinct general members of $\mcP$.
We have $T|_S = \Gamma + \Delta$, where
\[
\Delta = (\ell_1 = \ell_2 = w + g_3 = 0).
\]
By the assumption that $\pi (\Gamma) \not\subset B$, we have $g_3 \ne 0$, in other words, $\Delta \ne \Gamma$.
Consider the scheme-theoretic intersection
\begin{equation} \label{eq:GamDelint}
\Gamma \cap \Delta = (\ell_1 = \ell_2 = w = g_3 = 0) \subset \mbP (1, 1, 1, 1, 3),
\end{equation}
which consists of $3$ points counting with multiplicity.

\begin{Rem} \label{rem:GamDel}
We have $\Gamma \cap \Delta = \pi^{-1} (\pi (\Gamma) \cap B)$ set-theoretically.
A singular point $\msp \in X$ necessary satisfies $\pi (\msp) \in B$, so that a point $\msp \in \Gamma \cap \Sing X$ is contained in $\Gamma \cap \Delta$.
\end{Rem}

We choose and fix a point $\msp \in \Gamma$ which is a $cA_k$ point of $X$, where $k = 1, 2$.
We choose homogeneous coordinates $x, y, z, t, w$ of $\mbP (1, 1, 1, 1, 3)$ so that $X$ is defined by
\begin{equation} \label{eq:SDeq}
- w^2 + x^4 f_2 + x^3 f_3 + x^2 f_4 + x f_5 + f_6 = 0,
\end{equation}
where $f_i = f_i (y, z, t)$ is a homogeneous polynomial of degree $i$, and $\msp = (1\!:\!0\!:\!0\!:\!0\!:\!0)$.
This choice of coordinates is possible by \cite[Theorem A]{Pae21}.
Replacing coordinates $y, z, t$ (while fixing the point $\msp = (1\!:\!0\!:\!0\!:\!0\!:\!0)$), we may assume $\pi (\Gamma) = (z = t = 0) \subset \mbP^3$, that is, $\ell_1 = z$ and $\ell_2 = t$.
Then we can write
\begin{equation} \label{eq:SDGamma}
\Gamma = (z = t = w - (\lambda x^2 y + \mu x y^2 + \nu y^3) = 0) \subset \mbP (1, 1, 1, 1, 3),
\end{equation}
for some $\lambda, \mu, \nu \in \mbC$.
Note that at least one of $\lambda, \mu, \nu$ is nonzero since $\pi (\Gamma) \not\subset B$.
Note also that
\begin{equation} \label{eq:SDgamma}
x^4 f_2 (y, 0, 0) + x^3 f_3 (y, 0, 0) + \cdots + f_6 (y, 0, 0) = (\lambda x^2 y + \mu x y^2 + \nu y^3)^2
\end{equation}
since $\Gamma \subset X$.
With this choice of coordinates, the pencil $\mcP$ is generated by $z$ and~$t$.
The point $\msp$ is a multiplicity $1$ (resp.\ $2$, resp.\ $3$) point of $\Gamma \cap \Delta$ if and only if $\lambda \ne 0$ (resp.\ $\lambda = 0$ and $\mu \ne 0$, resp.\ $\lambda = \mu = 0$ and $\nu \ne 0$).
Let $S$ be a general member of $\mcP$.
We analyze the type of $G (S, \msp, \Gamma)$.

\begin{Lem} \label{lem:extGSD1}
Under the notation and assumption as above, let $\msp \in \Gamma$ be a $cA_1$ point of $X$.
Then the following assertions hold.
\begin{enumerate}
\item If $\mult_{\msp} \Gamma \cap \Delta = 1$, then $G (S, \msp, \Gamma)$ is of type $A_{1, 1}$.
\item If $\mult_{\msp} \Gamma \cap \Delta = 2$, then $G (S, \msp, \Gamma)$ is of type $A_{1, 1}$ or $A_{3, 2}$.
\item If $\mult_{\msp} \Gamma \cap \Delta = 3$, then $G (S, \msp, \Gamma)$ is of type $A_{1, 1}$, $A_{3, 2}$ or $A_{5, 3}$.
\end{enumerate}
\end{Lem}

\begin{proof}
We choose homogeneous coordinates of $\mbP := \mbP (1, 1, 1, 1, 3)$ as above, that is, $\msp = (1\!:\!0\!:\!0\!:\!0\!:\!0)$, $X$ is defined by the equation \eqref{eq:SDeq}, and $\Gamma$ is the one given in \eqref{eq:SDGamma}.
In the following, we assume that a general member $S \in \mcP$ is cut out by the equation $t = \theta z$ for a general $\theta \in \mbC$.
We set $U_x := \mbP (1, 1, 1, 1, 3) \setminus (x = 0)$ and, by a slight abuse of notation, we identify $U_x$ with the affine space $\mbA^4$ with affine coordinates $y, z, t, w$.

We first prove (1).
Suppose $\mult_{\msp} \Gamma \cap \Delta = 1$, that is, $\lambda \ne 0$.
In view of the equation \eqref{eq:SDgamma}, we can write
\[
f_2 = \lambda^2 y^2 + y \ell (z, t) + q (z, t),
\]
where $\ell (z, t)$ and $q (z, t)$ are linear and quadratic forms, respectively.
Note that at least one of $\ell (z, t)$ and $q (z, t)$ is nonzero since $\msp \in X$ is of type $cA_1$.
By eliminating the variable $t = \theta z$, $S$ is defined in $\mbP (1, 1, 1, 1, 3)$ by the equation
\[
- w^2 + x^4 (\lambda^2 y^2 + \zeta y z + \eta z^2) + x^3 f_3 (y, z, \theta z) + \cdots = 0,
\]
where $\zeta, \eta \in \mbC$ satisfy $(\zeta, \eta) \ne (0,0)$.
By setting $x = 1$, $S \cap U_x$ is isomorphic to the hypersurface in $\mbA^3_{y, z, w}$ defined by
\[
- w^2 + \lambda^2 y^2 + \zeta y z + \eta z^2 + h_{\ge 3} (y, z) = 0,
\]
where $h_{\ge 3} (y, z)$ is a polynomial of order at least $3$.
The point $\msp = (1\!:\!0\!:\!0\!:\!0\!:\!0)$ corresponds to the origin and the curve $\Gamma$ is defined in $\mbA^3$ by the equations
\begin{equation} \label{eq:GammacA1m1}
z = w - \lambda y - \mu y^2 - \nu y^3 = 0.
\end{equation}
It is easy to see that $\msp \in S$ is of type $A_1$ and the proper transform via the blowup of $\msp \in S$ of $\Gamma$ intersects the exceptional divisor transversally at one point.
It follows that $G (S, \msp, \Gamma)$ is of type $A_{1, 1}$.

We prove (2).
Suppose $\mult_{\msp} \Gamma \cap \Delta = 2$, that is, $\lambda = 0$ and $\mu \ne 0$.
We can write
\[
f_2 = y \ell (z, t) + q (z, t),
\]
where $\ell (z, t)$ and $q (z, t)$ are linear and quadratic forms, respectively.
Let $\zeta, \eta \in \mbC$ be such that $\ell (z, \theta z) = \zeta z$ and $q (z, \theta z) = \eta z^2$.
Then, by eliminating $t$ and setting $x = 1$, $S \cap U_x$ is isomorphic to the hypersurface in $\mbA^3_{y, z, w}$ defined by the equation
\begin{equation} \label{eq:ScA1m2}
- w^2 + \zeta y z + \eta z^2 + z (h_2 (y, z) + h_3 (y, z)) +\mu^2 y^4 + h_{\ge 5} (y, z) = 0,
\end{equation}
where $h_2, h_3 \in \mbC [y, z]$ are homogeneous polynomials of degree $2, 3$, respectively, and $h_{\ge 5} \in \mbC [y, z]$ is a polynomial of order at least $5$.
The curve $\Gamma$ is defined in $\mbA^3$ by the equations
\begin{equation} \label{eq:GammacA1m2}
z = w - \mu y^2 - \nu y^3 = 0.
\end{equation}

Suppose that $\ell (z, t) \ne 0$.
Then $\zeta \ne 0$, and it is easy to see that $\msp \in S$ is of type $A_1$ and $G (S, \msp, \Gamma)$ is of type $A_{1, 1}$.

Suppose that $\ell (z, t) = 0$.
Then $\zeta = 0$ and $\eta \ne 0$.
The exceptional divisor $E$ of the blow-up $\tilde{S} \to S$ of $S$ at $\msp$ is isomorphic to the hypersurface
\[
(- w^2 + \eta z^2 = 0) \subset \tilde{\mbP}^2 := \mbP^2_{y, z, w},
\]
where the polynomial $-w^2 + \eta z^2$ is the least degree part of \eqref{eq:ScA1m2}.
It follows that $E = E_+ + E_-$, where $E_{\pm} = (w \pm \sqrt{\eta} z = 0)$ are smooth rational curves which meet only at $\msq := (1\!:\!0\!:\!0) \in \tilde{\mbP}^2$.
By \eqref{eq:GammacA1m2}, the proper transform $\tilde{\Gamma} \subset \tilde{S}$ of $\Gamma$ intersect $E$ only at $\msq$.
The point $\tilde{\msp}$ is in the $y$-chart, denoted by $\tilde{S}_y$, of the blow-up $\tilde{S} \to S$, and $\tilde{S}_y$ is the hypersurface in $\tilde{\mbA}^3 := \mbA^3_{y, z, w}$ defined by
\[
- w^2 + \eta z^2 + z y (h_2 (1, z) + y h_3 (1, z)) + \mu^2 y^2 + h_{\ge 5} (y, y, z)/y^2 = 0.
\]
Note that $\tilde{\msp}$ corresponds to the origin, and it is an ordinary double point of $\tilde{S}$.
Moreover, on $\tilde{S}_y$, we have
\[
\begin{split}
\tilde{\Gamma} &=  (z = w - \mu y - \nu y^2 = 0) \subset \tilde{\mbA}^3, \\
E_{\pm} &= (y = w - \sqrt{\eta} z = 0) \subset \tilde{\mbA}^3.
\end{split}
\]
Let $\hat{S} \to \tilde{S}$ be the blowup of $\tilde{S}$ at $\tilde{\msp}$ with exceptional curve $F$.
Then $\hat{S}$ is smooth along $F$,
\[
F \cong (-w^2 + \eta z^2 + \mu^2 y^2 = 0) \subset \hat{\mbP}^2 = \mbP^2_{y, z, w}
\]
 is a smooth rational curve, and $\hat{\Gamma} \cap F = \{(1\!:\!0\!:\!\mu)\}$, $F \cap \hat{E}_{\pm} = \{(0\!:\!1\!:\! \pm \sqrt{\eta})\}$, $\hat{\Gamma} \cap \hat{E}_{\pm} = \emptyset$.
 This shows that $G (S, \msp, \Gamma)$ is of type $A_{3, 2}$.

We prove (3).
Suppose $\mult_{\msp} \Gamma \cap \Delta = 3$, that is, $\lambda = \mu = 0$ and $\nu \ne 0$.
We can write
\[
f_2 = y \ell (z, t) + q (z, t)
\]
as before.
If $\ell (z, t) \ne 0$, then $G (S, \msp, \Gamma)$ is of type $A_{1, 1}$ as in the previous case.
We assume $\ell (z, t) = 0$.
Then $q (z, t) \ne 0$, and by eliminating $t$ and setting $x = 1$, $S \cap U_x$ is isomorphic to the hypersurface in $\mbA^3_{y, z, w}$ defined by the equation
\[
-w^2 + \eta z^2 + z h (y, z) + \nu^2 y^6 = 0,
\]
where $\eta \in \mbC$ is nonzero and $h (y, z)$ is a polynomial consisting of monomials of degree $2, 3, 4, 5$.
The curve $\Gamma$ is defined in $\mbA^3$ by the equations
\begin{equation} \label{eq:GammacA1m3}
z = w - \nu y^3 = 0.
\end{equation}
We write
\[
h (y, 0) = \alpha_2 y^2 + \alpha_3 y^3 + \alpha_4 y^4 + \alpha_5 y^5,
\]
where $\alpha_i \in \mbC$.
Let $\tilde{S} \to S$ be the blow-up of $S$ at $\msp$ with exceptional divisor $E$.
We see that $E = E_+ + E_-$, where
\[
E_{\pm} := (w \pm \sqrt{\eta} z = 0) \subset \tilde{\mbP}^2 = \mbP^2_{y, z, w},
\]
are smooth rational curves which meet only at $\tilde{\msp} := (1\!:\!0\!:\!0) \in \tilde{\mbP}^2$.
Moreover the proper transform $\tilde{\Gamma} \subset \tilde{S}$ of $\Gamma$ meet $E_{\pm}$ only at $\tilde{\msp}$.
The $y$-chart $\tilde{S}_y$ of the blow-up $\tilde{S} \to S$ is the hypersurface in $\tilde{\mbA}^3 := \mbA^3_{y, z, w}$ defined by
\[
- w^2 + \eta z^2 + z y (\alpha_2 + \alpha_3 y + \alpha_4 y^2 + \alpha_5 y^3 + \cdots) + \nu^2 y^4 = 0,
\]
where the omitted terms consist of monomials divisible by $z$.
The point $\tilde{\msp}$ corresponds to the origin of $\tilde{\mbA}^3$.
On this chart $\tilde{S}_y$, we have
\[
\begin{split}
\tilde{\Gamma} &= (z = w  - \nu y^2 = 0) \subset \tilde{\mbA}^3, \\
E_{\pm} &= (y = w \pm \sqrt{\eta} z = 0) \subset \tilde{\mbA}^3.
\end{split}
\]
Now let $\hat{S} \to \tilde{S}$ be the blow-up of $\tilde{S}$ at $\tilde{\msp}$ with exceptional divisor
\[
F \cong (- w^2 + \eta z^2 + \alpha_2 z y = 0) \subset \hat{\mbP}^2 := \mbP^2_{y, z, w}.
\]
The proper transform $\hat{\Gamma} \subset \hat{S}$ of $\Gamma$ intersects $F$ only at the point $\hat{\msp} := (1\!:\!0\!:\!0) \in \hat{\mbP}^2$.
We denote by $\hat{S}_y$ the $y$-chart of the blow-up $\hat{S} \to \tilde{S}$ at $\tilde{\msp}$, which is the hypersurface in $\hat{\mbA}^3$ defined by
\[
- w^2 + \eta z^2 + z (\alpha_2 + \alpha_3 y + \alpha_4 y^2 + \alpha_5 y^3 + \cdots) + \nu^2 y^2 = 0.
\]
The point $\hat{\msp}$ corresponds to the origin of $\hat{\mbA}^3$, and $\hat{S}$ has an ordinary double point at $\hat{\msp}$.
Note that the proper transforms $\hat{E}_{\pm} \subset \hat{S}$ of $E_{\pm}$ are both disjoint from the $y$-chart $\hat{S}_y$.

Suppose that $\alpha_2 \ne 0$.
In this case $\tilde{S}$ has an ordinary double point at $\tilde{\msp}$.
We see that $\hat{\Gamma}$ intersects $F$ transversally at $\hat{\msp}$, $F \cap \hat{E}_{\pm} = \{(0\!:\!1\!:\!\pm \sqrt{\eta})\}$, and $\hat{E}_{\pm} \cap \hat{\Gamma} = \emptyset$.
This shows that $G (S, \msp, \Gamma)$ is of type $A_{3, 2}$.

Suppose that $\alpha_2 = 0$.
In this case we have $F = F_+ + F_-$, where
\[
F_{\pm} \cong (w \pm \sqrt{\eta} z = 0) \subset \hat{\mbP}^2,
\]
and we have $F_+ \cap F_- \cap \hat{\Gamma} = \{\hat{\msp}\}$.
On the $y$-chart $\hat{S}_y$, we have
\[
\begin{split}
\hat{\Gamma} &= (z = w - \nu y = 0) \subset \hat{\mbA}^3, \\
F_{\pm} &= (y = w \pm \sqrt{\eta} z = 0) \subset \hat{\mbA}^3.
\end{split}
\]
Let $\breve{S} \to \hat{S}$ be the blow-up of $\hat{S}$ at $\hat{\msp}$ with exceptional divisor $G$.
We have an isomorphism
\[
G \cong (- w^2 + \eta z^2 + \alpha_3 z y + \nu^2 y^2 = 0) \subset \breve{\mbP}^2_{y, z, w},
\]
and hence $G$ is a smooth rational curve.
Let $\breve{\Gamma}$, $\breve{E}_{\pm}$, and $\breve{F}_{\pm}$ be proper transforms of $\Gamma$, $E_{\pm}$, and $F_{\pm}$ on $\breve{S}$, respectively.
We have $G \cap \breve{\Gamma} = \{(1\!:\!0\!:\!\nu)\}$ and $G \cap \breve{F}_{\pm} = \{(0\!:\!1\!:\!\pm \sqrt{\eta})\}$.
This shows that $G (S, \msp, \Gamma)$ is of type $A_{5, 3}$.
This completes the proof.
\end{proof}

\begin{Lem} \label{lem:extGSD2}
Under the notation and assumption as above, let $\msp \in \Gamma$ be a $cA_2$-point of $X$.
Then the following assertions hold.
\begin{enumerate}
\item If $\mult_{\msp} \Gamma \cap \Delta = 1$, then $G (S, \msp, \Gamma)$ is of type $A_{2, 1}$.
\item If $\mult_{\msp} \Gamma \cap \Delta = 2$, then $G (S, \msp, \Gamma)$ is of type $A_{3, 2}$.
\item If $\mult_{\msp} \Gamma \cap \Delta = 3$, then $G (S, \msp, \Gamma)$ is of type $A_{3, 2}$ or $A_{5, 3}$.
\end{enumerate}
\end{Lem}

\begin{proof}
We use the same notation as in the previous proof.

We prove (1).
Suppose $\mult_{\msp} \Gamma \cap \Delta = 1$, that is, $\lambda \ne 0$.
Since $\msp \in X$ is of type $cA_2$, we have $f_2 = \lambda^2 y^2$ and $f_3 (0, z, t) \ne 0$ as a polynomial.
Then, by eliminating $t$ and setting $x = 1$, $S \cap U_x$ is isomorphic to the hypersurface in $\mbA^3_{y, z, w}$ defined by the equation
\[
- w^2 + \lambda^2 y^2 + \zeta z^3 + \alpha y z^2 + \beta y^2 z + h_{\ge 4} (y, z) = 0,
\]
where $\alpha, \beta, \zeta \in \mbC$ with $\zeta \ne 0$ and $h_{\ge 4} (y, z)$ is a polynomial of order at least $4$.
The point $\msp = (1\!:\!0\!:\!0\!:\!0\!:\!0)$ corresponds to the origin and the curve $\Gamma$ is defined in $\mbA^3$ by the equations \eqref{eq:GammacA1m1}.
Let $\tilde{S} \to S$ be the blow-up of $S$ at $\msp$ with exceptional divisor $E$.
Then $E = E_+ + E_-$, where
\[
E_{\pm} \cong (w \pm \lambda y = 0) \subset \tilde{\mbP}^2 := \mbP^2_{y, z, w}.
\]
The surface $\tilde{S}$ is smooth outside the singular locus of $E$.
The singular locus of $E$ is the intersection point $\msq$ of $E_+$ and $E_-$, where $\msq = (0\!:\!1\!:\!0) \in \tilde{\mbP}^2$.
The point $\msq$ corresponds to the origin of the $z$-chart $\tilde{S}_z$ which is the hypersurface in $\tilde{\mbA}^3 := \mbA^3_{y, z, w}$ defined by
\[
- w^2 + \lambda^2 y^2 + \zeta z + \alpha y + \beta y^2 z + h_{\ge 4} (y z, z)/z^2 = 0.
\]
It follows that $\tilde{S}$ is smooth at $\tilde{\msp}$.
Then, since $\Gamma$ is defined by the equations \eqref{eq:GammacA1m1}, we have $\tilde{\Gamma} \cap E = \{\msq'\}$, where $\msq' := (1\!:\!0\!:\!\lambda) \in \tilde{\mbP}^2$.
We have $\msq' \in E_-$ and $\msq' \notin E_+$.
This shows that $G (S, \msp, \Gamma)$ is of type $A_{2,1}$.

We prove (2).
Suppose $\mult_{\msp} \Gamma \cap \Delta = 2$, that is, $\lambda = 0$ and $\nu \ne 0$.
We see that $f_2$ is a square of a linear form in $z$ and $t$.
We may assume $f_2 = z^2$ by replacing $z$ and $t$.
By eliminating $t$ and setting $x = 1$, $S \cap U_x$ is isomorphic to the hypersurface in $\mbA^3_{y, z, w}$ defined by the equation
\[
-w^2 + z^2 + z h_2 (y, z) + \mu^2 y^4 + h_{\ge 5} (y, z) = 0,
\]
where $h_2 (y, z)$ is a homogeneous polynomial of degree $2$ and $h_{\ge 5} (y, z)$ is a polynomial of order at least $5$.
The curve $\Gamma$ is defined in $\mbA^3$ by the equations \eqref{eq:GammacA1m2}.
By repeating similar computations as in (2) of Lemma \ref{lem:extGSD1}, we see that $\msp \in S$ is of type $A_3$ and $G (S, \msp, \Gamma)$ is of type $A_{3, 2}$.

We prove (3).
Suppose $\mult_{\msp} \Gamma \cap \Delta = 3$, that is, $\lambda = \mu = 0$ and $\nu \ne 0$.
As in previous arguments, $S \cap U_x$ is isomorphic to the hypersurface in $\mbA^3_{y, z, w}$ defined by the equation
\[
-w^2 + z^2 + z h (y, z) + \nu^2 y^6 = 0,
\]
where $h (y, z)$ is a polynomial consisting of monomials of degree $2, 3, 4, 5$.
The curve $\Gamma$ is defined in $\mbA^3$ by the equations \eqref{eq:GammacA1m3}.
By the same argument as in the proof of Lemma~\ref{lem:extGSD1} (3), we conclude that $G (S, \msp, \Gamma)$ is of type $A_{3, 2}$ or $A_{5, 3}$.
\end{proof}

\begin{table}[tb]

\newlength{\myheightt}
\setlength{\myheightt}{0.6cm}

\caption{Types of $G (S, \msp, \Gamma)$}
\label{table:extG}
\centering
\begin{tabular}{|cc|c|c|}
\hline
\parbox[c][\myheightt][c]{0cm}{} $\msp \in X$ & $\mult_{\msp} \Gamma \cap \Delta$ & $G (S, \msp, \Gamma)$ & Correct.\ term \\
\hline
\parbox[c][\myheightt][c]{0cm}{} $cA_1$ & $1$ & $A_{1,1}$ & $\frac{1}{2}$  \\
\hline
\parbox[c][\myheightt][c]{0cm}{} $cA_1$ & $2$ & $A_{1, 1}$ or $A_{3, 2}$ & $\frac{1}{2}$ or $1$ \\
\hline
\parbox[c][\myheightt][c]{0cm}{} $cA_1$ & $3$ & $A_{1, 1}$ or $A_{3, 2}$ or $A_{5, 3}$ & $\frac{1}{2}$ or $1$ or $\frac{3}{2}$ \\
\hline
\parbox[c][\myheightt][c]{0cm}{} $cA_2$ & $1$ & $A_{2,1}$ & $\frac{2}{3}$ \\
\hline
\parbox[c][\myheightt][c]{0cm}{} $cA_2$ & $2$ & $A_{3, 2}$ & $1$ \\
\hline
\parbox[c][\myheightt][c]{0cm}{} $cA_2$ & $3$ & $A_{3, 2}$ or $A_{5, 3}$ & $1$ or $\frac{3}{2}$  \\
\hline
\end{tabular}
\end{table}

The results obtained by Lemmas~\ref{lem:extGSD1} and~\ref{lem:extGSD2} are summarized in Table~\ref{table:extG}, where the correction term given in the right-most column is the number $k (n-k+1)/(n+1)$ (corresponding to the extended graph of type $A_{n, k}$) appearing in Lemma~\ref{lem:selfintS}.

\begin{Prop} \label{prop:SDScurve}
Let $X$ be a  factorial sextic double solid with at worst terminal singularities of type $cA_1$ and $cA_2$.
Then no curve on $X$ is a maximal center.
\end{Prop}

\begin{proof}
By Lemma~\ref{lem:SDmostcurves}, it remains to consider a curve $\Gamma$ of degree $1$ such that $\Gamma \cap \Sing X \ne \emptyset$ and $\pi (\Gamma) \not\subset B$.

Let $\mcP \subset \left| - K_X \right|$ be the pencil as above, and let $S, T \in \mcP$ be general members.
By \cite[Lemma 10.4]{OkadaIII}, $S$ and $T$ are smooth outside $\Gamma \cap \Sing X$.
We have $T|_S = \Gamma + \Delta$, where $\Delta \ne \Gamma$ is also a curve of degree $1$.
By \eqref{eq:GamDelint}, we see that $\Gamma \cap \Delta$ consists of $3$ points counting with multiplicity.

We compute the self-intersection number $(\Gamma^2)_S$.
Let $\msp_1, \dots, \msp_l$ be the points in $\Gamma \cap \Sing (X)$.
By Remark~\ref{rem:GamDel}, these points are contained in $\Gamma \cap \Delta$.
We set $m_i = \mult_{\msp_i} \Gamma \cap \Delta$.
Then $l \le \sum m_i \le 3$.
Let $A_{n_i, k_i}$ be the type of $G (S, \msp_i, \Gamma)$ which is described in Lemmas~\ref{lem:extGSD1} and~\ref{lem:extGSD2} (see also Table~\ref{table:extG}).
Since $K_S = 0$ and $\Gamma \cong \mbP^1$, we have
\[
(\Gamma^2)_S = - 2 + \sum_{i = 1}^l \frac{k_i (n_i - k_i + 1)}{n_i + 1}
\]
by Lemma~\ref{lem:selfintS}.
By considering all the possible combinations of singular points $\msp_1, \dots, \msp_l$ satisfying $\sum m_i \le 3$, we have the upper bound
\[
(\Gamma^2)_S \le - 2 + \left(\frac{2}{3} + \frac{2}{3} + \frac{2}{3} \right) = 0,
\]
where this maximum (of the correction term corresponding to singular points) is attained when $\Gamma \cap \Delta$ consists of $3$ multiplicity $1$ points of $\Gamma \cap \Delta$ which are $cA_2$ singular points of $X$.
Since $T|_S = \Gamma + \Delta$, we have
\[
(\Gamma \cdot \Delta)_S = (T|_S \cdot \Gamma)_S - (\Gamma^2)_S \ge 1 = (- K_X \cdot \Delta).
\]
Thus, by Lemma~\ref{lem:exclcurve2}, the curve $\Gamma$ is not a maximal center.
\end{proof}

\subsection{Exclusion of $cA_1$ points and proof of Theorem~\ref{thm:SDS}}

\begin{Prop} \label{prop:SDSsingpt}
Let $X$ be a factorial sextic double solid with at worst terminal singularities of type $cA_1$.
Then no singular point on $X$ is a maximal center.
\end{Prop}

\begin{proof}
Let $\msp \in X$ be a $cA_1$ point.
Suppose that $\msp$ is a maximal center.
Then, by Theorems~\ref{thm:ineqcA1}, there exists a mobile linear system $\mcM \subset |\mcO_X (n)|$ on $X$ such that
\[
\mult_{\msp} (D_1 \cdot D_2) > 2 n^2,
\]
where $D_1, D_2$ are general members in  $\mcM$.
Let $\mcH$ be the linear subsystem of $|\mcO_X (1)|$ consisting of divisors vanishing at $\msp$.
Then the base locus of $\mcH$ consists of at most $2$ points (including $\msp$), and hence we can take a divisor $S$ in $\mcH$ which does not contain any component of the effective $1$-cycle $D_1 \cdot D_2$.
Then, we have
\[
2 n^2 = (S \cdot D_1 \cdot D_2) > 2 n^2.
\]
This is a contradiction and $\msp$ is not a maximal center.
\end{proof}

\begin{proof}[Proof of Theorem~\ref{thm:SDS}]
This is a consequence of Propositions~\ref{prop:SDSsmpt}, \ref{prop:SDScurve} and~\ref{prop:SDSsingpt}.
\end{proof}

\subsection{Sarkisov links centered at $cA_3$ singular points}
\label{sec:SDSlinks}

Let $X$ be a factorial sextic double solid with only terminal Gorenstein singularities and suppose that $X$ admits a $cA_3$ point $\msp \in X$.
Then $X$ is a hypersurface of degree $6$ in $\mbP (1, 1, 1, 1, 3)$ with homogeneous coordinates $x, y, z, t, w$ of weights $1, 1, 1, 1, 3$, respectively, and by \cite[Theorem A]{Pae21} the defining polynomial can be written as
\begin{equation} \label{eq:SDScA3eq}
f = -w^2 + x^4 t^2 + x^3 t f_2 + x^2 f_4 + x f_5 + f_6,
\end{equation}
where $f_i \in \mbC [y, z, t]$ is a homogeneous polynomial of degree $i$.
Moreover the $cA_3$ point $\msp$ is $(1\!:\!0\!:\!0\!:\!0\!:\!0) \in X$.
We set $g_4 = f_4 - \frac{1}{4} f_2^2$ and then
\begin{equation} \label{eq:SDScA3eqpm}
f_{\pm} = - w (w \pm (2 x^2 t + x f_2)) + x^2 g_4 + x f_5 + f_6.
\end{equation}
Let $X_{\pm} \subset \mbP (1, 1, 1, 1, 3)$ be the hypersurface defined by $f_{\pm} = 0$.
We see that $X$ is isomorphic to $X_{\pm}$ by the coordinate change $w \mapsto w \pm (x^2 t + \frac{1}{2} x f_2)$.

We set $U = X \setminus (x = 0) \cap X$ and $U_{\pm} = X_{\pm} \setminus (x = 0) \cap {X_{\pm}}$ which are affine open subsets of $X$ and $X_{\pm}$, respectively, and we have
\[
\begin{split}
U &= (- w^2 + t^2 + t f_2 + f_4 + f_5 + f_6 = 0) \subset \mbA^4_{y, z, t, w}, \\
U_{\pm} &= (- w (w \pm (t + f_2)) + g_4 + f_5 + f_6 = 0) \subset \mbA^4_{y, z, t, w}.
\end{split}
\]
Let $\varphi \colon Y \to X$ and $\varphi_{\pm} \colon Y_{\pm} \to X_{\pm}$ be the birational morphisms obtained as the weighted blow-up of $U$ and $U_{\pm}$ at the origin with weights $\wt (y, z, t, w) = (1, 1, 2, 2)$ and $\wt (y, z, t, w) = (1, 1, 1, 3)$, respectively.
According to the classification of divisorial contractions to $cA$ points \cite[Theorem 1.13]{Kaw03}, $\varphi, \varphi_{\pm}$ are all the divisorial contractions of $\msp \in X$ with discrepancy $1$.
We sometimes identify $\varphi_{\pm}$ with the composite $Y_{\pm} \xrightarrow{\varphi_{\pm}} X_{\pm} \cong X$.
We call $\varphi_+$, $\varphi$ and $\varphi_-$ the $(3, 1, 1, 1)$-, $(2, 2, 1, 1)$-and $(1, 3, 1, 1)$-Kawakita blow-ups of $\msp \in X$, respectively.

\subsubsection{$(3, 1, 1, 1)$-, and $(1, 3, 1, 1)$-Kawakita blow-ups}

We consider the $(3, 1, 1, 1)$-and $(1, 3, 1, 1)$-Kawakita blow-up $\varphi_{\pm}$ of $\msp \in X$, and show that there is a Sarkisov self-link initiated by $\varphi_{\pm}$ if $\varphi_{\pm}$ is a maximal extraction.

Let $\pi_{X_{\pm}} \colon X_{\pm} \ratmap \mbP (1, 1, 1, 3)$ be the projection to the coordinates $y, z, t, w$, which is defined outside the point $\msp = (1\!:\!0\!:\!0\!:\!0\!:\!0) \in X_{\pm}$.
The sections $y, z, t \in \mathrm{H}^0 (X_{\pm}, -K_{X_{\pm}})$ lift to sections $\tilde{y}, \tilde{z}, \tilde{t}$ in $\mathrm{H}^0 (Y_{\pm}, -K_{Y_{\pm}})$, and $w \in \mathrm{H}^0 (X_{\pm}, - 3 K_{X_{\pm}})$ lifts to a section $\tilde{w}$ in $\mathrm{H}^0 (Y_{\pm}, - 3 K_{Y_{\pm}})$.
These sections $\tilde{y}, \tilde{z}, \tilde{z}$ and $\tilde{w}$ define the composite $\pi_{Y_{\pm}} := \pi_{X_{\pm}} \circ \varphi_{\pm} \colon Y_{\pm} \ratmap \mbP (1, 1, 1, 3)$, and their common zero locus (on $Y_{\pm}$) is empty.
This is verified in the following way.
For the $\varphi_{\pm}$-exceptional divisor $E_{\pm}$, we have the  natural isomorphism
\[
E_{\pm} \cong (- w (\pm t + f_2) + g_4 = 0) \subset \mbP (1, 1, 1, 3) = \Proj \mbC [y, z, t, w],
\]
and the restriction of the sections $\tilde{y}, \tilde{z}, \tilde{t}, \tilde{w}$ on $E_{\pm}$ are $y, z, t, w$, respectively.
Thus their common zero loci on $E_{\pm}$, and hence on $Y_{\pm}$, are empty.
It follows that $\pi_{Y_{\pm}} \colon Y_{\pm} \to \mbP (1, 1, 1, 3)$ is a morphism.
The defining polynomial $f_{\pm}$ given in \eqref{eq:SDScA3eqpm} is quadratic with respect to $x$:
\[
f_{\pm} = x^2 (g_4 \mp w t) + x (f_5 \mp w f_2) + f_6 - w^2.
\]
We set
\[
D_4 := (w t + f_4 = 0)_{X_{\pm}}, \
D_5 := (w f_2 + f_5 = 0)_{X_{\pm}}, \
D_6 := (w^2 + f_6 = 0)_{X_{\pm}},
\]
and denote by $\tilde{D}_i$ the proper transform of $D_i$ on $Y_{\pm}$.
We see that $\pi_{Y_{\pm}}$ is a generically finite morphism of degree $2$ which contracts the subset $\Xi := \tilde{D}_4 \cap \tilde{D}_5 \cap \tilde{D}_6$ onto its image
\[
\pi_{Y_{\pm}} (\Xi) = (g_4 \mp w t = f_5 \mp w f_2 = f_6 -w^2 = 0) \subset \mbP (1, 1, 1, 3).
\]
Let $Y_{\pm} \xrightarrow{\psi_{\pm}} Z_{\pm} \xrightarrow{\pi_{Z_{\pm}}} \mbP (1, 1, 1, 3)$ be the Stein factorization of $\pi_{Y_{\pm}}$.
Then $\psi_{\pm}$ is the birational morphism defined by $\left|- r K_{Y_{\pm}} \right|$ for sufficiently large $r > 0$ and $\pi_{Z_{\pm}}$ is a double cover.
Note that $-K_{Y_{\pm}}$ is nef and big but not ample.
Let $\iota_{\pm} \colon Z_{\pm} \to Z_{\pm}$ be the biregular involution interchanging two points in fibers of the double cover $\pi_{Z_{\pm}}$, and consider the diagram
\begin{equation} \label{eq:SDScA3linkQ}
\xymatrix{
 & \ar[ld]_{\varphi_{\pm}} Y_{\pm} \ar[rd]^{\psi_{\pm}} \ar@{-->}[rrr]^{\tau_{\pm}} & & & \ar[ld]_{\psi_{\pm}} Y_{\pm} \ar[rd]^{\varphi_{\pm}} & \\
X_{\pm} & & Z_{\pm} \ar[r]^{\iota_{\pm}} & Z_{\pm} & & X_{\pm}}
\end{equation}
where $\tau_{\pm} := \psi^{-1}_{\pm} \circ \iota_{\pm} \circ \psi_{\pm}$.

\begin{Prop}
Let $X$ be a factorial sextic double solid $X$ with only terminal Gorenstein singularities and let $\msp \in X$ be a $cA_3$ point.
Let $\varphi_{\pm} \colon Y_{\pm} \to X$ be the $(3, 1, 1, 1)$-, $(1, 3, 1, 1)$-Kawakita blow-ups with center $\msp$.
Then one of the following holds.
\begin{enumerate}
\item $\psi_{\pm}$ is divisorial.
In this case $\varphi_{\pm}$ is not a maximal extraction.
\item $\psi_{\pm}$ is small.
In this case $\tau_{\pm}$ is the flop of $\psi_{\pm}$ and the diagram \eqref{eq:SDScA3linkQ} gives a Sarkisov self-link $\sigma_{\pm} \colon X \ratmap X$ initiated by $\varphi_{\pm}$.
\end{enumerate}
\end{Prop}

\begin{proof}
This follows from the above arguments and \cite[Lemma 3.2]{OkadaII}.
\end{proof}

\subsubsection{$(2, 2, 1, 1)$-Kawakita blow-up}

We consider the $(2, 2, 1, 1)$-Kawakita blow-up $\varphi \colon Y \to X$ of $\msp \in X$, and show that $\varphi$ is not a maximal extraction.

\begin{Prop}
Let $X$ be a factorial sextic double solid $X$ with only terminal Gorenstein singularities and let $\msp \in X$ be a $cA_3$ point.
Let $\varphi \colon Y \to X$ be the $(2, 2, 1, 1)$-Kawakita blow-up.
Then $\varphi$ is not a maximal extraction.
\end{Prop}

\begin{proof}
Let $\pi \colon X \ratmap \mbP (1, 1, 1, 2, 4)$ be the map defined by
\[
(x\!:\!y\!:\!z\!:\!t\!:\!w) \mapsto (y\!:\!z\!:\!t\!:\!x t\!:\!w t).
\]
Let $y, z, t, \alpha, \beta$ be the homogeneous coordinates of $\mbP (1, 1, 1, 2, 4)$ of weights $1, 1, 1, 2, 4$, respectively.
Then the image of $\pi$, denoted by $Z$, is the hypersurface in $\mbP (1, 1, 1, 2, 4)$ defined by
\[
h := - \beta^2 + \alpha^4 + \alpha^3 f_2 + \alpha^2 f_4 + \alpha t f_5 + t^2 f_6 = 0,
\]
where $h$ is obtained by replacing $x t$ by $\alpha$ and $w t$ by $\beta$ in $t^2 f$.
The sections $y, z, t, x t, w t$ lift to pluri-anticanonical sections on $Y$ and their common zero locus (on $Y$) is empty.
It follows that the induced map $\psi = \pi \circ \varphi \colon Y \ratmap Z$ is a morphism, and we have the commutative diagram:
\[
\xymatrix{
& \ar[ld]_{\varphi} Y \ar[rd]^{\psi} & \\
X \ar@{-->}[rr]_{\pi} & & Z}
\]
Note that $\psi$ is a birational morphism and it contracts the proper transform of the divisor $(t = 0)_X$ to the curve
\[
\mbP^1 \cong (t = \alpha = \beta = 0) \subset \mbP (1, 1, 1, 2, 4).
\]
Moreover $\psi$ is defined by $\left| - r K_Y \right|$ for a sufficiently large $r > 0$, and thus any curve contracted by $\psi$ is $K_Y$-trivial.
It follows that there are infinitely many curves on $Y$ which intersect $-K_Y$ non-positively and the $\psi$-exceptional divisor $E$ positively.
By \cite[Lemma 2.20]{OkadaII}, $\varphi$ is not a maximal extraction.
\end{proof}

\begin{Rem}
Let $X$ be a factorial sextic double solid with only terminal $cA_1, cA_2, cA_3$ points.
We can exclude smooth points and $cA_1$ points as a maximal center by Propositions~\ref{prop:SDSsmpt} and \ref{prop:SDSsingpt}.
Moreover, by the argument of Proposition~\ref{prop:SDScurve}, we can exclude curves on $X$ except possibly for curves of degree $1$ passing through a $cA_3$ point.
Therefore, for the proof of birational rigidity of $X$, it remains to show the following:
\begin{itemize}
\item Any curve of degree $1$ passing through a $cA_3$ point is not a maximal center.
\item Any $cA_2$ point is not a maximal center.
\item Any divisorial contraction to a $cA_3$ point with discrepancy at least $2$ (if it exists) is not a maximal extraction.
\end{itemize}
\end{Rem}

\subsection{Factorial sextic double solids} \label{sec:fac}

In this section, we provide a criterion for $\mbQ$-factoriality of sextic double solids with terminal $cA_1$ singularities.
If a $3$-fold~$X$ with only  terminal $cA_1$ singularities is a Fano, a hypersurface of $\mbP^4$, or a cyclic cover of $\mbP^3$, the Picard group is isomorphic to the 2nd integral cohomology because $\mathrm{H}^1 (X, \mcO_X) = \mathrm{H}^2 (X, \mcO_X) = 0$.
In this case, $\mbQ$-factoriality of $X$ can be determined by the following  global topological property.
\begin{Prop}\label{prop:Cl}
Let $X$ be a $3$-fold with only  terminal $cA_1$ singularities.
Suppose that  $\mathrm{H}^1 (X, \mcO_X) = \mathrm{H}^2 (X, \mcO_X) = 0$.
Then $X$ is  $\mbQ$-factorial if and only if
 \begin{equation}\label{eq:duality}
\dim_{\mbQ} (\mathrm{H}^2 (X, \mbQ)) = \dim_{\mbQ} (\mathrm{H}_4 (X, \mbQ)).
\end{equation}
\end{Prop}
\begin{proof}
It is enough to show that  \eqref{eq:duality} holds if and only if
\[\mathrm{Cl}(X)\otimes\mbQ\cong\mathrm{Pic}(X)\otimes\mbQ.\]
where $\mathrm{Cl}$  denotes the divisor class group  and $\mathrm{Pic}$ denotes the Picard group.
It follows from $\mathrm{H}^1 (X, \mcO_X) = \mathrm{H}^2 (X, \mcO_X) = 0$ that $\mathrm{H}^2 (X, \mbQ) \cong\mathrm{Pic}(X)\otimes\mbQ$.
Therefore, it is enough to show that $\mathrm{H}_4 (X, \mbQ) \cong\mathrm{Cl}(X)\otimes\mbQ$.

Let $\pi:\widetilde{X}\to X$ be a resolution of singularities and let $E_i$, $i=1,\ldots, r$, be its exceptional  prime divisors. We denote by $E$ the exceptional locus of $\pi$, i.e. $E=\cup E_i$. Then,  we can obtain the following commutative diagram:
\begin{equation}\label{eq:NT}
\xymatrix{
\mathrm{Cl}(\widetilde{X})\otimes\mbQ \ar[r]  \ar[d]_{\pi_*^C} & \mathrm{H}_4 (\widetilde{X}, \mbQ) \ar[d]^{\pi_*^H} \\
\mathrm{Cl}(X)\otimes\mbQ \ar@{->}[r]&  \mathrm{H}_4 (X, \mbQ)}
\end{equation}
from the natural transformation in  \cite[\S~5]{Ful75}.
Note that
\[\mathrm{Cl}(\widetilde{X})\otimes\mbQ \cong\left(\mathrm{Cl}(X)\otimes\mbQ\right)\bigoplus\left(\oplus_{i=1}^{r} \mbQ E_i\right)\]
and the homomorphism $\pi^C_*: \mathrm{Cl}(\widetilde{X})\otimes\mbQ\to \mathrm{Cl}(X)\otimes\mbQ$ given by the resolution morphism is surjective with the kernel generated by the exceptional divisors. Since $\widetilde{X}$ is a smooth projective $3$-fold,
\[\mathrm{Cl}(\widetilde{X})\otimes\mbQ\cong\mathrm{Pic}(\widetilde{X})\otimes\mbQ \cong\mathrm{H}^2 (\widetilde{X}, \mbQ) \cong\mathrm{H}_4 (\widetilde{X}, \mbQ),\]
where the last isomorphism follows from the Poincar\'e duality. There is an exact sequence
\[\ldots\longrightarrow \mathrm{H}_4(E,\mbQ) \longrightarrow \mathrm{H}_4(\widetilde{X},\mbQ) \longrightarrow \mathrm{H}_4(\widetilde{X}/E,\mbQ) \longrightarrow \mathrm{H}_3(E,\mbQ) \longrightarrow \ldots,\]
where $\widetilde{X}/E$  is  the quotient topological space of $\widetilde{X}$ by the subspace $E$ and it is topologically equivalent to $X$. The exact sequence shows that
the kernel of the homomorphism $\pi_*^H:\mathrm{H}_4 (\widetilde{X}, \mbQ)\to \mathrm{H}_4 (X, \mbQ)$ is  $\oplus_{i=1}^{r} \mbQ E_i$. This implies that the bottom horizontal homomorphism in \eqref{eq:NT} is injective.

Since the singularities of $X$ are only terminal  $cA_1$, we may assume that
each $E_i$ is a smooth rational surface. This implies that $\mathrm{H}_3(E_i,  \mbQ)=\mathrm{H}^1(E_i,  \mbQ)=0$. It then follows from a Mayer-Vietoris sequence that $\mathrm{H}_3(E,  \mbQ)=0$.
This vanishing shows that the homomorphism $\pi_*^H$ is surjective. Consequently, the  bottom horizontal homomorphism in  \eqref{eq:NT} is surjective, and hence it is an isomorphism.
\end{proof}

The duality \eqref{eq:duality} fails on singular varieties in general (for instance, see  Examples~\ref{ex:fac10} and~\ref{ex:fac8}).
The failure of the duality can be measured by the difference of their dimensions, so called, the \textit{defect} of $X$
\[
\delta_X := \dim_{\mbQ} (\mathrm{H}_4 (X, \mbQ)) - \dim_{\mbQ} (\mathrm{H}^2 (X, \mbQ)).
\]
With this notion, one may say that $X$ is $\mbQ$-factorial if and only if $\delta_X = 0$.

We  consider a double cover $X$ of $\mbP^3$ branched along a surface $B \subset \mbP^3$ of degree $2 r$ with only Du Val singularities.
We can regard $X$ as a hypersurface of degree $2 r$ in the weighted projective space $\mbP = \mbP (1, 1, 1, 1, r)$.
In general, the singular points on $X$ may have an effect on the integral (co)homology groups of $X$ (See \cite{Cle83}, \cite{Cyn02} and \cite{Dim90}).
However, it follows from a Lefschetz-type theorem (\cite[Proposition 1.4]{Mav99}) that $\Pic (X) \cong \Pic (\mbP)$, in particular, $\rank (\Pic (X)) = 1$.

Meanwhile, as aforementioned,  it follows from \cite[ Lemma 5.1]{Kawamata88} that  on Gorenstein terminal $3$-folds, factoriality is equivalent to $\mbQ$-factoriality.
Therefore, in our case, the double solid $X$ is $\mbQ$-factorial if and only if it is factorial.

It is not simple to compute the defect.
Fortunately, a method to compute some Hodge numbers for resolutions of double solids branched along surfaces with only Du Val singularities has been introduced by Rams \cite{Ram08} which has evolved from the paper of Clemens \cite{Cle83}.
The paper \cite{Ram08} formulates the difference between the Hodge numbers of the big resolution $Y$ and the Hodge numbers of $\mbP$.
Here the \textit{big resolution} $\rho \colon Y \to X$ is defined by the composition of blow ups:
\[
\rho = \rho_{k} \circ \cdots \circ \rho_1 \colon Y \to Y_0 = X,
\]
where $\rho_i \colon Y_i \to Y_{i-1}$ for $i = 1, \dots, k$, is the blow up with center $\Sing (Y_{i-1}) \ne \emptyset$ and $Y_{k} = Y$ is smooth. Note that all singularities and infinitely near singularities of $X$ are isolated double points, in particular, each $\rho_i$ is a blow up centered at closed points.
The Hodge numbers of $\mbP$ are considered with the Zariski sheaf of germs of $1$-forms $\overline{\Omega}^1_{\mbP} := \iota_* (\Omega^1_{\reg (\mbP)})$, where $\iota \colon \reg (\mbP) \to \mbP$ is the inclusion.

The defect of $X$ appears in a Hodge number of the big resolution $Y$ of $X$, i.e.,
\[
h^{1, 1} (Y) = h^1 (\mbP, \overline{\Omega}^1_{\mbP}) + \mu_X + \delta_X,
\]
where $\mu_X$ is the number of components of the exceptional locus of the big resolution $\rho$ (\cite[Theorem 4.1]{Ram08}).
Note that $h^1 (\mbP, \overline{\Omega}^1_{\mbP}) = \rank (\Pic (\mbP)) = 1$.

If the double solid $X$ has only terminal $cA_1$ singularities, then the surface $B$ has only $A_n$ singularities.
Conversely, the double solid branched along a surface with only $A_n$ singularities has only terminal $cA_1$ singularities.
Meanwhile, a non-cyclic Du Val singular point of $B$ yields a terminal $cA_2$ singular point on $X$.

From now on, for our purpose we suppose that the double solid $X$ has only terminal $cA_1$ singularities unless otherwise mentioned.

We will adopt the method given by \cite[Theorem 4.1]{Ram08} to compute the defect $\delta_X$. To do so,
we put
\[
\mcH_X = \{\, H \in \mathrm{H}^0 (\mbP^3, \mcO_{\mbP^3} (3 r - 4)) \mid \text{$H$ fulfills the condition $(\mathbf{A})$} \,\},
\]
where the condition $(\mathbf{A})$ is as follows:
\begin{quote}
For each singular point $p \in \Sing (B)$, suppose the surface $B$ is given by
\[
x_{1,p}^{m+1} + x_{2,p}^2 + x_{3,p}^2 + f (x_{1,p}, x_{2,p}, x_{3,p}) = 0
\]
for some positive integer $m$ locally around $p$, where $x_{i,p}$'s are analytic coordinates centered at the point $p$ and $f (x_{1,p}, x_{2,p}, x_{3,p})$ is a polynomial of order strictly greater than $1$ with respect to the weights $\wt (x_{1,p}) = \frac{1}{m+1}$, $\wt (x_{2,p}) = \wt (x_{3,p}) = \frac{1}{2}$.
For the singular point $p$, impose the conditions:
\[
\begin{split}
(\mathrm{C}_{p,0}) & \hspace{10mm} H (p) = 0; \\
(\mathrm{C}_{p,j}) & \hspace{10mm} \text{$\frac{\prt^j H}{\prt x_{1,p}^j} (p) = 0$ for $1 \le j \le \lceil \frac{m}{2} \rceil -1$}.
\end{split}
\]
\end{quote}

\begin{Thm} \label{thm:facA}
The defect of the double solid $X$ with only terminal $cA_1$ singularities is given by
\[
\delta_X = \dim (\mcH_X) - (h^0 (\mbP^3, \mcO_{\mbP^3} (3 r - 4)) - \mu_X).
\]
\end{Thm}

\begin{proof}
The following have been verified in \cite[Theorem in \S 1.4 and Theorem in \S 2.3]{Dol82}:
\[
\begin{split}
& \text{$\mathrm{H}^i (\mbP, \mcO_{\mbP} (-2 r)) = 0$ for $i \le 3$}; \\
& \text{$\mathrm{H}^i (\mbP, \mcO_{\mbP} (-4 r)) = 0$ for $i \le 3$}; \\
& \mathrm{H}^2 (\mbP, \overline{\Omega}^1_{\mbP}) = 0; \\
& \text{$\mathrm{H}^i (\mbP, \overline{\Omega}^1_{\mbP} \otimes \mcO_{\mbP} (-2 r)) = 0$ for $i = 1, 2, 3$}.
\end{split}
\]
Then \cite[Theorem 4.1]{Ram08} yields
\[ \dim_{\mbQ} (\mathrm{H}^2 (Y, \mbQ)) =h^{1,1}(Y)=1+ \dim (\mcH_X) - h^0 (\mbP^3, \mcO_{\mbP^3} (3 r - 4)) +2 \mu_X.\]
As shown in the proof of Proposition~\ref{prop:Cl},  $\dim_{\mbQ} (\mathrm{H}_4 (X, \mbQ)) =\mathrm{rank}(\mathrm{Cl}(X))$. Moreover, $\dim_{\mbQ} (\mathrm{H}^2 (X, \mbQ)) =\mathrm{rank}(\mathrm{Pic}(X))=1$. Therefore,
\[\begin{split}
\delta_X&=\dim_{\mbQ} (\mathrm{H}_4 (X, \mbQ)) - \dim_{\mbQ} (\mathrm{H}^2 (X, \mbQ))\\
&=\mathrm{rank}(\mathrm{Cl}(X)) -1\\
&=\mathrm{rank}(\mathrm{Cl}(Y)) -\mu_X-1\\
&=\mathrm{rank}(\mathrm{Pic}(Y)) -\mu_X-1\\
&=\dim_{\mbQ} (\mathrm{H}^2 (Y, \mbQ)) -\mu_X-1\\
&=\dim (\mcH_X) - (h^0 (\mbP^3, \mcO_{\mbP^3} (3 r - 4)) - \mu_X).
\end{split}\]
\end{proof}


Theorem~\ref{thm:facA} reduces the topological problem to a relatively simpler problem concerning the linear system $|\mcO_{\mbP^3} (3 r - 4)|$ and $0$-dimensional subschemes in $\mbP^3$.

\begin{Cor} \label{cor:faccondmu}
The double solid $X$ is factorial if and only if the condition~$(\mathbf{A})$  imposes linearly independent $\mu_X$ conditions on $\mathrm{H}^0 (\mbP^3, \mcO_{\mbP^3} (3 r - 4))$.
\end{Cor}

\begin{Cor}
Suppose that the surface $B$ has singularities
\[
a_1 A_1 +a_2A_2+ \cdots + a_m A_m,
\]
where $a_i$ is the number of singular points of type $A_i$ on $B$.
If the double solid $X$ is $\mbQ$-factorial, then
\[
\sum_{n=1}^m a_n \lceil \frac{n}{2} \rceil \le h^0 (\mbP^3, \mcO_{\mbP^3} (3 r - 4)) = \frac{1}{2} (3 r - 1)(3 r - 2)(r - 1).
\]
\end{Cor}

\begin{proof}
A  singular point $p$  of type $A_n$ yields $\lceil \frac{n}{2} \rceil$ conditions  $(\mathrm{C}_{p,j})$, $0 \le j \le \lceil \frac{m}{2} \rceil -1$,  on $\mathrm{H}^0 (\mbP^3, \mcO_{\mbP^3} (3 r - 4))$.
Therefore, the statement follows from Corollary~\ref{cor:faccondmu}.
\end{proof}

Following the argument based on a Lefschetz-type theorem as in \cite[Example 1.5]{CP10}, we can construct an example of a non-$\mbQ$-factorial double solid $X$ branched along a surface $B$ with only terminal $cA_k$ singularities.

\begin{Ex}\label{ex:fac9}
Let $B$ be the surface of degree $2 r$, where $r \ge 2$, defined by
\[
x f_{2 r - 1} (x, y, z) + t^{2 r -3} x (x^2 + y^2) + t ^{2 r  - 2} z^2 = 0 \subset \mbP^3,
\]
where $f_{2 r - 1}\in\mathbb{C}[x,y,z]$ is a general  homogeneous polynomial of degree~$2 r - 1$.
Then the surface $B$ has $2 r - 3$ singular points of type $A_1$, $2 r -1$ singular points of type $A_{2 r - 3}$, and one singular point of type $D_4$.
The double solid $X$ branched along the surface $B$ has $4 r - 4$ terminal singular points of type $cA_1$ and one terminal singular point of type $cA_2$.
Furthermore, it can be defined by the weighted homogeneous equation
\[
w^2 = x f_{2 r - 1} (x, y, z) + t^{2 r -3} x (x^2 + y^2) + t^{2 r - 2} z^2 \subset \mbP (1, 1, 1, 1, r).
\]
The hyperplane section defined by $x = 0$ splits into two divisors given by the equation
\[
(w + z t^{r-1})(w - z t^{r-1}) = 0,
\]
each of which is a non-$\mbQ$-Cartier divisor.
\end{Ex}

We can construct the following non-$\mbQ$-factorial double solid with only terminal~$cA_1$ singularities based on the fact that a variety with a small resolution is not $\mbQ$-factorial. In this example, Theorem~\ref{thm:facA} confirms that its defect is positive, with the aid of  the Cayley-Bacharach theorem.

\begin{Ex}\label{ex:fac10}
Let $V$ be the smooth $3$-fold of bidegree $(2, r)$, $r\geq 2$, in $\mbP^1 \times \mbP^3$ defined by the bihomogeneous equation
\[
f_r (x, y, z) u^2 +2 t^r uv+ h_r (x, y, z) v^2 = 0,
\]
where $f_r$ and $h_r$ are general  homogeneous polynomials of degree $r$  in $\mathbb{C}[x, y,z]$.
In addition, we denote the natural projection of $V$ to $\mbP^3$ by $\pi \colon V \to \mbP^3$.
The system of equations
\[
f_r (x, y, z) = t = h_r (x, y, z) = 0
\]
defines exactly $r^2$ points in $\mbP^3$.
The $3$-fold $V$ then has exactly $r^2$ curves $C_i$, $i = 1, \dots, r^2$, such that $H \cdot C_i = 0$, where $H$ is a divisor cut by a hypersurface of bidegree $(0, 1)$ in $\mbP^1 \times \mbP^3$.
The projection $\pi$ has degree $2$ outside the points $\pi (C_i)$.
The model
\[
\Proj \left( \bigoplus_{n \ge 0} \mathrm{H}^0 (V, \mcO_V (n H)) \right)
\]
of $V$ is the double cover $X$ of $\mbP^3$ branched along the surface $B$ defined by
\[
t^{2r} - f_r (x, y, z) h_r (x, y, z) = 0.
\]
It has $r^2$ terminal singular points of type $cA_1$ each of which comes from each curve $C_i$.
The morphism $\phi_{|m H|} \colon V \to X$ given by the complete linear system of bidegree $(0, m)$ on $V$ with sufficiently large $m$ contracts these $r^2$ curves to $cA_1$ points of $X$.
Therefore, it is a small morphism of $V$ onto $X$ and hence the double cover~$X$ cannot be $\mbQ$-factorial.

Meanwhile, we can also verify that  $X$
cannot be $\mathbb{Q}$-factorial from the view point of
Theorem~ \ref{thm:facA}.
The surface $B$ has $r^2$ singular points of type $A_{2r-1}$ on the plane~$\Pi$ defined by $t=0$.
These singular points $p_i$, $i=1,\cdots, r^2$, on $B$ impose conditions ($\mathrm{C}_{p_i,r-1}$)
\[\frac{\partial^{r-1} H}{\partial t^{r-1}}(p_i)=0 \]
on $\mathrm{H}^0(\mathbb{P}^3, \mathcal{O}_{\mathbb{P}^3}(3r-4))$.
These conditions can be also regarded as vanishing conditions at $p_i$'s
for  $\mathrm{H}^0(\Pi, \mathcal{O}_\Pi(2r-3))$ since
$$\mathrm{H}^0(\mathbb{P}^3, \mathcal{O}_{\mathbb{P}^3}(3r-4))=\bigoplus_{k=0}^{3r-4}t^{3r-4-k}\mathbb{C}[x,y,z]_{k},$$
where $\mathbb{C}[x,y,z]_{k}$ is the space of homogenous polynomials of degree $k$ in variables~$x, y, z$.
The points $p_i$'s are the intersection points of the two plane curves of degree $r$
defined by $f_r(x, y, z)=0$ and $g_r(x,y,z)=0$ on $\Pi$. It therefore follows from the Cayley-Bacharach theorem (\cite[Theorem~CB4]{EGH96}) that the vanishing conditions at $p_i$'s  on  $\mathrm{H}^0(\Pi, \mathcal{O}_\Pi(2r-3))$ are linearly dependent.  Therefore, the conditions $$\{  (\mathrm{C}_{p_i,r-1}) | i=1, \cdots r^2\}$$
 are linearly dependent on $\mathrm{H}^0(\mathbb{P}^3, \mathcal{O}_{\mathbb{P}^3}(3r-4))$.
It then follows from Corollary~\ref{cor:faccondmu} that $X$ is not $\mathbb{Q}$-factorial, i.e.,
$\delta_X>0$.
\end{Ex}

\begin{Ex}
Let $B$ be the surface of degree $2 r$, where $r \ge 2$, defined by
\[
L_1 (x, y, z) L_2 (x, y, z) \cdots L_{2 r} (x, y, z) + t^{2 r - 1} z = 0 \subset \mbP^3,
\]
where $L_1 (x, y, z), \dots, L_{2 r} (x, y, z)$ are  general linear forms.

Then the surface $B$ has exactly $r (2 r - 1)$ singular points and they are $A_{2 r - 2}$ singularities.
These singular points correspond to the intersection points $p_{n m}$ of the lines $L_n (x, y, z)$ and $L_m (x, y, z)$ on the plane $t = 0$, $1 \le n< m \le 2 r$.
The double solid $X$ branched along the surface $B$ also has $r (2 r -1)$ terminal singular points of type $cA_1$.

Let $F_{l} (x, y, z)$ be a general from of degree $l$ in $x, y, z$.
For $0 \le j \le r-2$ and $1 \le n< m \le 2 r$, put
\[
H_{n, m, j} (x, y, z) = t^j F_{r - 2 - j} (x, y, z) \frac{L_1 (x, y, z) L_2 (x, y, z) \cdots L_{2 r} (x, y, z)}{L_n (x, y, z) L_m (x, y, z)}.
\]
It is a form of degree $3 r - 4$.
Observe that for fixed $n, m, j$,
\begin{gather*}
\frac{\prt^j H_{n, m, j}}{\prt t^j} (p_{n m}) \ne 0; \\
\frac{\prt^k H_{n, m, j}}{\prt t^k} (p_{n m}) = 0 \text{ if $k \ne j$}; \\
\frac{\prt^k H_{n, m, j}}{\prt t^k} (p_{n' m'}) = 0 \text{ if $(n', m') \ne (n, m)$}.
\end{gather*}
This implies that the $r (2 r - 1)(r -1)$ conditions $\{(\mathrm{C}_{p_{n m}, j})\}$ impose linearly independent conditions on $\mathrm{H}^0 (\mbP^3, \mcO_{\mbP^3} (3 r - 4))$.
Therefore, Corollary~\ref{cor:faccondmu} implies that the double solid $X$ is factorial.
\end{Ex}

$\mbQ$-factoriality of nodal double solids has been extensively studied in \cite{CP10}, \cite{Che09}, \cite{HP07} and \cite{Klo22}.
Their results are mainly based on the result of \cite{Cle83}.
As aforementioned, in this section, we use the result of \cite{Ram08} that is applicable to the cases with wider range of singularities.
Before we proceed, let us mention that some results in \cite{CP10}, \cite{Che09}, \cite{HP07} and \cite{Klo22} remain true if we allow only singularities slightly worse than nodes.

\begin{Thm}
Let $X$ be the double solid branched along a surface $B$ of degree~$2r$.
Suppose that $B$ has at worst $A_1$ and $A_2$ singularities.
\begin{enumerate}
\item If the number of singular points of $X$ is less than $r (2 r -1)$, then $X$ is factorial.
\item If the number of singular points of $X$ is $r (2 r -1)$ and $X$ is not $\mbQ$-factorial, then the surface $B$ is defined by an equation of the form
\[
f_r (x, y, z, t)^2 + f_1 (x, y, z, t) f_{2 r - 1} (x, y, z, t) = 0,
\]
where $f_i (x, y, z, t)\in\mathbb{C}[x,y,z,t]$ is a homogeneous polynomial of degree $i$.
\end{enumerate}
\end{Thm}

\begin{proof}
The statements have been verified under the condition that $B$ has only~$A_1$ singularities  in  \cite[Theorem B]{CP10}, \cite[Theorem 3]{Che09}, \cite[Theorems 4.3 and 5.3]{HP07}, and \cite[Theorem 4.6]{Klo22}.
However, as $A_1$ singularities, singular points of type $A_2$ on $B$ produce only one condition $(\mathrm{C}_{p,0})$ in the condition $(\mathbf{A})$.
From the view point of factoriality, we do not have to distinguish between $A_1$ and $A_2$.
Therefore, the proofs of \cite[Theorem 3]{Che09}  and \cite[Theorem 4.6]{Klo22} work verbatim for the case when $B$ allows $A_2$ singularities also.
\end{proof}

\begin{Ex} \label{ex:fac8}
Let $B$ be the surface of degree $2 r$, where $r \ge 2$, defined by
\[
t^{2 r} + x h_{2 r - 1} (x, y, z) = 0 \subset \mbP^3,
\]
where $h_{2 r - 1}$ is a general homogeneous polynomial of degree $2 r - 1$.
Then the surface $B$ has exactly $2 r - 1$ singular points which are of type $A_{2 r - 1}$.
The double solid $X$ branched along the surface $B$ also has $2 r - 1$ terminal singular points of type $cA_1$.
As before, this can be defined by the weighted homogeneous equation
\[
w^2 = t^{2 r} + x h_{2 r - 1} (x, y, z) \subset \mbP (1, 1, 1, 1, r).
\]
The hyperplane section defined by $x = 0$ splits into two divisors given by the equation
\[
(w + t^r)(w - t^r) = 0,
\]
each of which is a non-$\mbQ$-Cartier divisor.
\end{Ex}

Example~\ref{ex:fac8} leads to the following conjecture.

\begin{Conj} \label{conj:fac}
Let $X$ be the double solid branched along a sextic surface with at worst  $A_m$ singularities.
If the number of singular points on $X$ is at most four, then $X$ is factorial.
\end{Conj}

However, due to our limited understanding of  Du Val singularities of sextic surfaces, we can currently confirm it only under a specific additional  condition.

\begin{Thm}
Conjecture~\ref{conj:fac} holds good under the assumption that $B$ allows at most four and at worst $A_4$ singularities.
\end{Thm}

\begin{proof}
We assume that $B$ has four singular points denoted by $p_1, \dots, p_4$. 
The proof for the case of a smaller number of singular points is nearly identical.

For each index $i$, the singular point $p_i$ is of type $A_{m_i}$ with $m_i \leq 4$. Each point imposes linear conditions given by
\[
\begin{split}
& (\mathrm{C}_{i,0}): \hspace{10mm} H (p_i) = 0; \\
& (\mathrm{C}_{i,1}): \hspace{10mm} \frac{\prt H}{\prt x_{1, p_i}} (p_i) = 0 \text{ if } m_i=3 \text{ or } 4
\end{split}
\]
on the linear system of quintic hypersurfaces $H$ in $\mbP^3$.
To establish our claim, it suffices to demonstrate that for any given condition $(\mathrm{C}_{i,j})$, there exists a quintic hypersurface $H$ that violates the condition $(\mathrm{C}_{i,j})$ while simultaneously satisfying the remaining conditions.

We employ homogeneous coordinates $(x\!:\!y\!:\!z\!:\! t)$ for $\mbP^3$.
Consider a fixed condition  $(\mathrm{C}_{i,j})$.
Without loss of generality, we assume  that $p_i = (1\!:\!0\!:\!0\!:\!0)$ for convenience, and set $i = 4$.

In case when $m_4 \ge 2$, assume that the Zariski tangent planes at $p_4$ are defined by $y = 0$ and $z = 0$, respectively.
We may regard $x_{1, p_4}$ as $t$. 
\smallskip

\paragraph{\textit{Case 1}} The four singular points are not collinear.

Select a line passing through $p_4$ that includes the maximum number of singular points $p_1, p_2, p_3$. Such a line accommodates at least two and at most three singular points since the four singular points are not collinear. 
 We may therefore assume that $p_3$ lies outside the line.
 For $k=1, 2$, let $H_{k3}$ be a general hyperplane that contains $p_k$ and $p_3$. Note that $H_{k3}$ does not vanish at $p_4$.
 
 Set 
 \[
P_1 = H_{13}^2H_{23}^3, \
P_2 = tH_{13}^2H_{23}^2. \
\]
Then,
\[
\det
\begin{pmatrix}
P_1 & P_2 \\[1mm]
\frac{\prt P_1}{\prt t} & \frac{\prt P_2}{\prt t}
\end{pmatrix}
(p_4) =(H_{13}^4H_{23}^5)(p_4)
 \ne 0.
\]
Therefore, for the given condition $(\mathrm{C}_{4,j})$, there exists a point $(\mu_1\!:\!\mu_2) \in \mbP^1$ such that between two values
\[
\begin{split}
& (0) \hspace{10mm} (\mu_1 P_1 + \mu_2 P_2)(p_4); \\
& (1) \hspace{10mm} \frac{\prt}{\prt t} (\mu_1 P_1+ \mu_2 P_2)(p_4),
\end{split}
\]
the value $(j)$ is non-zero while the other is zero. 
Then the quintic
\[
\mu_1 P_1 + \mu_2 P_2
\]
violates  the condition $(\mathrm{C}_{4,j})$ but satisfies all the other conditions.

\smallskip

\paragraph{\textit{Case 2}} The four singular points are collinear. 

In this case, if $m_k, m_\ell\geq 2$, then the intersection line of Zariski tangent planes at $p_k$ must be distinct from that  of Zariski tangent planes at $p_\ell$.
Otherwise, the surface $B$ would be singular along the intersection line.

\smallskip

\paragraph{\textit{Subcase 1-1}}  $m_4\leq 2$.

Suppose $m_k\leq 2$ for some $k\leq 3$, say $k=3$. Take a general hyperplane $H_i$ passing through $p_i$.
Then the quintic $H_1^2H_2^2H_3$ violates the condition $(\mathrm{C}_{4,0})$ but satisfies all other conditions. 
Therefore, we may assume that $m_1, m_2, m_3\geq 3$.
Given that the intersection lines of Zariski tangent planes at $p_1$, $p_2$, $p_3$ are all different, there exists a hyperplane $H_k$ for some $k\in \{1,2,3\}$ such that $H_k(p_4)\ne 0$, $H_k(p_k)=0$, and $\frac{\partial H_k}{\partial x_{1, p_k}} (p_k) = 0$. Assume, for instance, that $k=3$. Then the quintic $H_1^2H_2^2H_3$ fails the condition $(\mathrm{C}_{4,0})$ but satisfies all other conditions.

\smallskip

\paragraph{\textit{Subcase 1-2}}   $m_4\geq 3$.

Among the three points $p_1$, $p_2$, and $p_3$, we either have $m_k, m_\ell \leq 2$ or ${m_k, m_\ell \geq 3}$. 
For both cases, we may assume that $k=2$ and $\ell=3$.

Let $H_1$ be a general hyperplane passing through $p_1$.

If $m_2$, $m_3 \leq 2$, then for each $i=2,3$, we consider a general hyperplane $H_i$ passing through $p_i$.

If $m_2$, $m_3 \geq 3$, then there exists a hyperplane $H_i$ for $i=2, 3$, such that ${H_i(p_4) \neq 0}$, $H_i(p_i) = 0$, and $\frac{\partial H_k}{\partial x_{1, p_i}}(p_i) = 0$. 
We select such a general hyperplane~$H_i$.

We 
set 
 \[
P_1 = H_{1}^2H_{2}^2H_3, \
P_2 = tH_{1}^2H_{2}H_3.\
\]

Then,
\[
\det
\begin{pmatrix}
P_1 & P_2 \\[1mm]
\frac{\prt P_1}{\prt t} & \frac{\prt P_2}{\prt t}
\end{pmatrix}
(p_4) =(H_1^4H_2^3H_3^2)(p_4)
 \ne 0.
\]
As in Case~1, for the given condition $(\mathrm{C}_{4,j})$, there is a point $(\mu_1\!:\!\mu_2) \in \mbP^1$ such that the quintic
\[
\mu_1 P_1 + \mu_2 P_2
\]
violates the condition $(\mathrm{C}_{4,j})$ but satisfies all the other conditions.
\end{proof}

\section{Prime Fano $3$-fold weighted complete intersections}

We consider some prime Fano $3$-fold weighted complete intersections with only $cA_1$ points together with terminal quotient singular points, and prove their birational rigidity.
We make the setting precise.

\begin{Setting} \label{setting:FanoWCI}
Let $X$ be a prime Fano $3$-fold which is either a weighted hypersurface of index $1$ in one of the families listed in Table~\ref{table:Fanohyp}, or a weighted complete intersection of codimension $2$ and index $1$ in one of the families listed in Table~\ref{table:FanoWCI}.
Let $\mbP$ be the ambient weighted projective space of $X$.
We assume that $X$ is quasi-smooth along the singular locus $\Sing (\mbP)$ of $\mbP$.
\end{Setting}

\begingroup
\scalefont{0.86}

\begin{table}[h]
\renewcommand{\arraystretch}{1.15}
\begin{center}
\caption{Fano $3$-fold weighted hypersurfaces of index $1$}
\label{table:Fanohyp}
\begin{tabular}{cccc|cccc}
\hline
No & $X_d \subset \mbP (1, a_1, \dots, a_4)$ & $-K_X^3$ & Case & No & $X_d \subset \mbP (1, a_1, \dots, a_4)$ & $-K_X^3$ & Case  \\
\hline
\rowcolor{lightgray}
6 & $X_8 \subset \mbP (1,1,1,2,4)$ & $1$ & $\diamondsuit$ & 57 & $X_{24} \subset \mbP (1, 3, 4, 5, 12)$ & $1/30$ & $\heartsuit$ \\
10 & $X_{10} \subset \mbP (1,1,1,3,5)$ & $2/3$ & $\diamondsuit$ & 58 & $X_{24} \subset \mbP (1, 3, 4, 7, 10)$ & $1/35$ & $\heartsuit$ \\
\rowcolor{lightgray}
11 & $X_{10} \subset \mbP (1,1,2,2,5)$ & $1/2$ & $\diamondsuit$ & 59 & $X_{24} \subset \mbP (1, 3, 6, 7, 8)$ & $1/42$ & $\heartsuit$ \\
14 & $X_{12} \subset \mbP (1,1,1,4,6)$ & $1/2$ & $\diamondsuit$ & 60 & $X_{24} \subset \mbP (1, 4, 5, 6, 9)$ & $1/45$ & $\heartsuit$ \\
\rowcolor{lightgray}
15 & $X_{12} \subset \mbP (1, 1, 2, 3, 6)$ & $1/3$ & $\heartsuit$ & 61 & $X_{25} \subset \mbP (1, 4, 5, 7, 9)$ & $5/252$ & $\heartsuit$ \\
16 & $X_{12} \subset \mbP (1, 1, 2, 4, 5)$ & $3/10$ &  & 62 & $X_{26} \subset \mbP (1, 1, 5, 7, 13)$ & $2/35$ & $\diamondsuit$ \\
\rowcolor{lightgray}
17 & $X_{12} \subset \mbP (1, 1, 3, 4, 4)$ & $1/4$ &  & 63 & $X_{26} \subset \mbP (1, 2, 3, 8, 13)$ & $1/24$ & $\diamondsuit$ \\
18 & $X_{12} \subset \mbP (1, 2, 2, 3, 5)$ & $1/5$ & $\clubsuit$ & 64 & $X_{26} \subset \mbP (1, 2, 5, 6, 13)$ & $1/30$ & $\diamondsuit$ \\
\rowcolor{lightgray}
19 & $X_{12} \subset \mbP (1, 2, 3, 3, 4)$ & $1/6$ & $\heartsuit$ & 65 & $X_{27} \subset \mbP (1, 2, 5, 9, 11)$ & $3/110$ & $\clubsuit$ \\
21 & $X_{14} \subset \mbP (1, 1, 2, 4, 7)$ & $1/4$ & $\diamondsuit$ & 66 & $X_{27} \subset \mbP (1, 5, 6, 7, 9)$ & $1/70$ & $\heartsuit$ \\
\rowcolor{lightgray}
22 & $X_{14} \subset \mbP (1, 2, 2, 3, 7)$ & $1/6$ & $\diamondsuit$ & 67 & $X_{28} \subset \mbP (1, 1, 4, 9, 14)$ & $1/18$ & $\diamondsuit$ \\
25 & $X_{15} \subset \mbP (1, 1, 3, 4, 7)$ & $5/28$ & & 68 & $X_{28} \subset \mbP (1, 3, 4, 7, 14)$ & $1/42$ & $\heartsuit$ \\
\rowcolor{lightgray}
26 & $X_{15} \subset \mbP (1, 1, 3, 5, 6)$ & $1/6$ & $\clubsuit$ & 69 & $X_{28} \subset \mbP (1, 4, 6, 7, 11)$ & $1/66$ & $\heartsuit$ \\
27 & $X_{15} \subset \mbP (1, 2, 3, 5, 5)$ & $1/10$ & $\heartsuit$ & 70 & $X_{30} \subset \mbP (1, 1, 4, 10, 15)$ & $1/20$ & $\diamondsuit$ \\
\rowcolor{lightgray}
28 & $X_{15} \subset \mbP (1, 3, 3, 4, 5)$ & $1/12$ & $\heartsuit$ & 71 & $X_{30} \subset \mbP (1, 1, 6, 8, 15)$ & $1/24$ & $\diamondsuit$ \\
29 & $X_{16} \subset \mbP (1, 1, 2, 5, 8)$ & $1/5$ & $\diamondsuit$ & 72 & $X_{30} \subset \mbP (1, 2, 3, 10, 15)$ & $1/30$ & $\heartsuit$ \\
\rowcolor{lightgray}
30 & $X_{16} \subset \mbP (1, 1, 3, 4, 8)$ & $1/6$ & $\diamondsuit$ & 73 & $X_{30} \subset \mbP (1, 2, 6, 7, 15)$ & $1/42$ & $\diamondsuit$ \\
31 & $X_{16} \subset \mbP (1, 1, 4, 5, 6)$ & $2/15$ & & 74 & $X_{30} \subset \mbP (1, 3, 4, 10, 13)$ & $1/52$ & $\clubsuit$ \\
\rowcolor{lightgray}
32 & $X_{16} \subset \mbP (1, 2, 3, 4, 7)$ & $2/21$ & $\clubsuit$ & 75 & $X_{30} \subset \mbP (1, 4, 5, 6, 15)$ & $1/60$ & $\heartsuit$ \\
34 & $X_{18} \subset \mbP (1, 1, 2, 6, 9 )$ & $1/6$ & $\diamondsuit$ & 76 & $X_{30} \subset \mbP (1, 5, 6, 8, 11)$ & $1/88$ & $\heartsuit$ \\
\rowcolor{lightgray}
35 & $X_{18} \subset \mbP (1, 1, 3, 5, 9)$ & $2/15$ & $\diamondsuit$ & 77 & $X_{32} \subset \mbP (1, 2, 5, 9, 16)$ & $1/45$ & $\diamondsuit$ \\
36 & $X_{18} \subset \mbP (1, 1, 4, 6, 7)$ & $3/28$ & & 78 & $X_{32} \subset \mbP (1, 4, 5, 7, 16)$ & $1/70$ & $\heartsuit$ \\
\rowcolor{lightgray}
37 & $X_{18} \subset \mbP (1, 2, 3, 4, 9)$ & $1/12$ & $\diamondsuit$ & 79 & $X_{33} \subset \mbP (1, 3, 5, 11, 14)$ & $1/70$ & $\clubsuit$ \\
38 & $X_{18} \subset \mbP (1, 2, 3, 5, 8)$ & $3/40$ & & 80 & $X_{34} \subset \mbP (1, 3, 4, 10, 17)$ & $1/60$ & $\diamondsuit$ \\
\rowcolor{lightgray}
41 & $X_{20} \subset \mbP (1, 1, 4, 5, 10)$ & $1/10$ & $\heartsuit$ & 81 & $X_{34} \subset \mbP (1, 4, 6, 7, 17)$ & $1/84$ & $\heartsuit$ \\
42 & $X_{20} \subset \mbP (1, 2, 3, 5, 10)$ & $1/15$ & $\heartsuit$ & 82 & $X_{36} \subset \mbP (1, 1, 5, 12, 18)$ & $1/30$ & $\diamondsuit$ \\
\rowcolor{lightgray}
43 & $X_{20} \subset \mbP (1, 2, 4, 5, 9)$ & $1/18$ & $\clubsuit$ & 83 & $X_{36} \subset \mbP (1, 3, 4, 11, 18)$ & $1/66$ & $\diamondsuit$ \\
44 & $X_{20} \subset \mbP (1, 2, 5, 6, 7)$ & $1/21$ & $\heartsuit$ & 84 & $X_{36} \subset \mbP (1, 7, 8, 9, 12)$ & $1/168$ & $\heartsuit$ \\
\rowcolor{lightgray}
45 & $X_{20} \subset \mbP (1, 3, 4, 5, 8)$ & $1/24$ & $\heartsuit$ & 85 & $X_{38} \subset \mbP (1, 3, 5, 11, 19)$ & $2/165$ & $\diamondsuit$ \\
46 & $X_{21} \subset \mbP (1, 1, 3, 7, 10)$ & $1/10$ &  & 86 & $X_{38} \subset \mbP (1, 5, 6, 8, 19)$ & $1/120$ & $\heartsuit$ \\
\rowcolor{lightgray}
47 & $X_{21} \subset \mbP (1, 1, 5, 7, 8)$ & $3/40$ &  & 87 & $X_{40} \subset \mbP (1, 5, 7, 8, 20)$ & $1/140$ & $\heartsuit$ \\
48 & $X_{21} \subset \mbP (1, 2, 3, 7, 9)$ & $1/18$ & $\clubsuit$ & 88 & $X_{42} \subset \mbP (1, 1, 6, 14, 21)$ & $1/42$ & $\heartsuit$ \\
\rowcolor{lightgray}
49 & $X_{21} \subset \mbP (1, 3, 5, 6, 7)$ & $1/30$ & $\heartsuit$ & 89 & $X_{42} \subset \mbP (1, 2, 5, 14, 21)$ & $1/70$ & $\heartsuit$ \\
51 & $X_{22} \subset \mbP (1, 1, 4, 6, 11)$ & $1/12$ & $\diamondsuit$ & 90 & $X_{42} \subset \mbP (1, 3, 4, 14, 21)$ & $1/84$ & $\heartsuit$ \\
\rowcolor{lightgray}
52 & $X_{22} \subset \mbP (1, 2, 4, 5, 11)$ & $1/20$ & $\diamondsuit$ & 91 & $X_{44} \subset \mbP (1, 4, 5, 13, 22)$ & $1/130$ & $\diamondsuit$ \\
53 & $X_{24} \subset \mbP (1, 1, 3, 8, 12)$ & $1/12$ & $\heartsuit$ & 92 & $X_{48} \subset \mbP (1, 3, 5, 16, 24)$ & $1/120$ & $\heartsuit$ \\
\rowcolor{lightgray}
54 & $X_{24} \subset \mbP (1, 1, 6, 8, 9)$ & $1/18$ & $\clubsuit$ & 93 & $X_{50} \subset \mbP (1, 7, 8, 10, 25)$ & $1/280$ & $\heartsuit$ \\
55 & $X_{24} \subset \mbP (1, 2, 3, 7, 12)$ & $1/21$ & $\diamondsuit$ & 94 & $X_{54} \subset \mbP (1, 4, 5, 18, 27)$ & $1/180$ & $\heartsuit$ \\
\rowcolor{lightgray}
56 & $X_{24} \subset \mbP (1, 2, 3, 8, 11)$ & $1/22$ & $\clubsuit$ & 95 & $X_{66} \subset \mbP (1, 5, 6, 22, 33)$ & $1/330$ & $\heartsuit$
\end{tabular}
\end{center}
\end{table}
\endgroup

\begingroup
\scalefont{0.95}

\begin{table}[h]
\renewcommand{\arraystretch}{1.15}
\begin{center}
\caption{Fano $3$-fold WCIs of codimension $2$ and index $1$}
\label{table:FanoWCI}
\begin{tabular}{ccc|ccc}
\hline
No & $X_{d_1, d_2} \subset \mbP (a_0, \dots, a_5)$ & $-K_X^3$ & No & $X_{d_1, d_2} \subset \mbP (a_0, \dots, a_5)$ & $-K_X^3$ \\
\hline
\rowcolor{lightgray}
8 & $X_{4, 6} \subset \mbP (1, 1, 2, 2, 2, 3)$ & $1$ & 59 & $X_{12, 14} \subset \mbP (1, 4, 4, 5, 6, 7)$ & $1/20$ \\
14 & $X_{6, 6} \subset \mbP (1, 2, 2, 2, 3, 3)$ & $1/2$ & 60 & $X_{12, 14} \subset \mbP (2, 3, 4, 5, 6, 7)$ & $1/30$ \\
\rowcolor{lightgray}
20 & $X_{6, 8} \subset \mbP (1, 2, 2, 3, 3, 4)$ & $1/3$ & 64 & $X_{12, 16} \subset \mbP (1, 2, 5, 6, 7, 8)$ & $2/35$ \\
24 & $X_{6, 10} \subset \mbP (1, 2, 2, 3, 4, 5)$ & $1/4$ & 71 & $X_{14, 16} \subset \mbP (1, 4, 5, 6, 7, 8)$ & $1/30$ \\
\rowcolor{lightgray}
31 & $X_{8, 10} \subset \mbP (1, 2, 3, 4, 4, 5)$ & $1/6$ & 75 & $X_{14, 18} \subset \mbP (1, 2, 6, 7, 8, 9)$ & $1/24$ \\
37 & $X_{8, 12} \subset \mbP (1, 2, 3, 4, 5, 6)$ & $2/15$ & 76 & $X_{12, 20} \subset \mbP (1, 4, 5, 6, 7, 10)$ & $1/35$ \\
\rowcolor{lightgray}
45 & $X_{10, 12} \subset \mbP (1, 2, 4, 5, 5, 6)$ & $1/10$ & 78 & $X_{16, 18} \subset \mbP (1, 4, 6, 7, 8, 9)$ & $1/42$ \\
47 & $X_{10, 12} \subset \mbP (1, 3, 4, 4, 5, 6)$ & $1/12$ & 84 & $X_{18, 30} \subset \mbP (1, 6, 8, 9, 10, 15)$ & $1/120$ \\
\rowcolor{lightgray}
51 & $X_{10, 14} \subset \mbP (1, 2, 4, 5, 6, 7)$ & $1/12$ & 85 & $X_{24, 30} \subset \mbP (1, 8, 9, 10, 12, 15)$ & $1/180$
\end{tabular}
\end{center}
\end{table}
\endgroup

\begin{Thm} \label{thm:WCI}
Let $X$ be as in Setting~\ref{setting:FanoWCI} and suppose that $X$ has at worst $cA_1$ points besides terminal quotient singular points.
Then $X$ is birationally rigid.
\end{Thm}

\begin{Rem} \label{rem:Setting}
Let $X \subset \mbP$ be as in Setting~\ref{setting:FanoWCI}.
Then the subset $X \cap \Sing (\mbP)$ consists of the terminal cyclic quotient singular points of $X$ coming from the ambient space since $X$ is quasi-smooth along $\Sing (\mbP)$.
Any other singular point of $X$, if it exists, is contained in a smooth locus of $\mbP$ so that it is a hypersurface---hence Gorenstein---singular point of $X$.
In Theorem~\ref{thm:WCI}, these hypersurface singular points are assumed to be $cA_1$ points.
\end{Rem}

We fix some notations.
Let $\mbP = \mbP (a_0, \dots, a_n)$ be a weighted projective space with homogeneous coordinates $x_0, \dots, x_n$ of weights $a_0, \dots, a_n$, respectively.
For homogeneous polynomials $g_1, \dots, g_N$ in variables $x_0, \dots, x_n$, we denote by
\[
(g_0 = \cdots = g_N = 0) \subset \mbP (a_0, \dots, a_n)
\]
the closed subscheme defined by the homogeneous ideal $(g_1, \dots, g_N)$.

\begin{Def} \label{def:qsm}
Under the above notation, let $V = (f_1 = \cdots = f_m = 0) \subset \mbP$ be a closed subscheme, where $f_1, \dots, f_m$ are homogeneous polynomials in variables $x_0, \dots, x_n$.
The \textit{quasi-smooth locus} $\Qsm (V)$ of $V$ is defined to be the image of the smooth locus of $C^*_V := C_V \setminus \{o\}$ under the natural map $\mbA^{n+1} \setminus \{o\} \to \mbP$, where
\[
C_V := \Spec \mbC [x_0, \dots, x_n]/(f_1, \dots, f_m)
\]
is the \textit{affine cone} of $V$ and $o$ is the origin of $\mbA^{n+1}$.
For a subset $S \subset V$, we say that $V$ is \textit{quasi-smooth along} $S$ if $S \subset \Qsm (V)$.
We say that $V$ is \textit{quasi-smooth} if it is quasi-smooth along $V$, that is, $V = \Qsm (V)$.
\end{Def}

For $X \subset \mbP$ as in Setting~\ref{setting:FanoWCI} and $s \in \{x, y, z, t, w\}$, we set
\[
U_s = (s \ne 0) \cap X := X \setminus \left( (s = 0) \cap X\right).
\]

\subsection{Exclusion of curves}

\begin{Prop} \label{prop:WCIcurve}
Let $X$ be as in Setting~\ref{setting:FanoWCI}.
Then no curve on $X$ is a maximal center.
\end{Prop}

\begin{proof}
Suppose that there is a curve $\Gamma \subset X$ which is a maximal center.
By \cite{Kawamata}, there exists no divisorial contraction centered along a curve through a terminal quotient singular point.
Hence $\Gamma$ does not pass through a terminal quotient singular point.
In particular we have $(-K_X \cdot \Gamma) \ge 1$ since $-K_X$ is a Cartier divisor on an open subset containing $\Gamma$ and thus $(-K_X \cdot \Gamma)$ is a positive integer.
On the other hand we have $(-K_X)^3 \le 1$.
By Lemma~\ref{lem:exclcurve1}, $\Gamma$ cannot be a maximal center.
This is a contradiction and the proof is completed.
\end{proof}

\subsection{Exclusion of smooth points}

The aim of this section is to show that no smooth point on $X$ is a maximal center for $X$ as in Setting~\ref{setting:FanoWCI}.
It may be possible to say that this follows from the same arguments as in \cite[Section 2.1]{CP17} and \cite[Section 7]{OkadaI}.
However we reproduce the proofs in a general setting mainly because these arguments will also be useful in the exclusion of $cA_1$ points as a maximal center (see Section~\ref{sec:WCIGorsing}).

\begin{Def} \label{def:isol}
Let $V$ be a closed subscheme of a weighted projective space $\mbP (a_0, \dots, a_n)$ with homogeneous coordinates $x_0, \dots, x_n$ of weights $a_0, \dots, a_n$, respectively, and let $\msp \in V$ be a point.
We say that homogeneous polynomials $g_1, \dots, g_N$ in variables $x_0, \dots, x_n$ (resp.\ effective Weil divisors $D_1, \dots, D_N$ on $V$) {\it isolate} $\msp$ or they are $\msp$-{\it isolating polynomials} (resp.\ $\msp$-{\it isolating divisors}) if the set
\[
(g_1 = \cdots = g_N = 0) \cap V \quad (\text{resp.} \Supp D_1 \cap \cdots \cap \Supp D_N)
\]
is a finite set of points including $\msp$.
For a positive integer $e$, we also say that $g_1, \dots, g_N$ (resp.\  $D_1, \dots, D_N$) are $\msp$-{\it isolating polynomials of degree at most $e$} (resp.\ $\msp$-{\it isolating divisors of degree at most $e$}) if $\max \{ \deg g_1, \dots, \deg g_N \} \le e$ (resp.\ $D_i \in |\mcO_V (e_i)|$ for $i = 1, \dots, N$ and $\max \{e_1, \dots, e_N\} \le e$).
\end{Def}

\begin{Rem} \label{rem:isoldivs}
Let $V$ be as in Definition~\ref{def:isol} and let $\msp \in V$ be a point.
Suppose that there are $\msp$-isolating divisors of degree at most $e$.
Let $\Gamma \subset V$ be a closed subset such that $\msp \in \Gamma$ and any component of $\Gamma$ is of positive dimension. Then we can take a divisor $D \in |\mcO_V (e')|$ for some $e' \le e$ such that $\msp \in \Supp D$ and no component of $\Gamma$ is contained in $\Supp D$.
\end{Rem}

\begin{Lem} \label{lem:isoldivs}
Let $V$ be a complete intersection of codimension $c \in \{1, 2\}$ in a weighted projective space $\mbP (a_0, \dots, a_n)$ with homogeneous coordinates $x_0, \dots, x_n$ of weights $a_0, \dots, a_n$, respectively.
\begin{enumerate}
\item Let $\msp \in V$ be a point.
\begin{enumerate}
\item There exist $\msp$-isolating divisors of degree at most
\[
\max \{\, \lcm (a_i, a_j) \mid i, j \in \{0, 1, \dots, n\} \, \}.
\]
\item If there is $k$ such that $\msp \notin \cap_{j \ne k} (x_j = 0) \cap V$, then there exist $\msp$-isolating divisors of degree at most
\[
\max \{\, \lcm (a_i, a_j) \mid i, j \in \{0, 1, \dots, n\} \setminus \{k\} \, \}.
\]
\item If $c = 2$ and there are $k_1, k_2$ such that $\msp \notin \cap_{j \ne k_1, k_2} (x_j = 0)$, then there exist $\msp$-isolating divisors of degree at most
\[
\max \{\, \lcm (a_i, a_j) \mid i, j \in \{0, 1, \dots, \} \setminus \{k_1, k_2\} \,\}.
\]
\end{enumerate}
\item Let $\msp \in U_{x_i} = (x_i \ne 0) \cap V$ be a point.
\begin{enumerate}
\item There exist $\msp$-isolating divisors of degree at most
\[
\max \{\, \lcm (a_i, a_j) \mid j \in \{0, 1, \dots, n\} \,\}.
\]
\item If there is $k \ne i$ such that $\msp \notin \cap_{j \ne k} (x_j = 0) \cap V$, then there exists $\msp$-isolating divisors of degree at most
\[
\max \{\, \lcm (a_i, a_j) \mid j \in \{0, 1, \dots, n\} \setminus \{k\} \, \}.
\]
\item If $c = 2$ and there are $k_1, k_2 \ne i$ such that $\msp \notin \cap_{j \ne k_1, k_2} (x_j = 0) \cap V$, then there exists $\msp$-isolating divisors of degree at most
\[
\max \{\, \lcm (a_i, a_j) \mid j \in \{0, 1, \dots, n\} \setminus \{k_1, k_2\} \,\}.
\]
\end{enumerate}
\end{enumerate}
\end{Lem}

\begin{proof}
We set $U_{x_j} = (x_j \ne 0) \cap V$.
(1) follows from (2) since $V$ is covered by the~$U_{x_j}$.

We prove (2).
We can write $\msp = (\alpha_0\!:\!\cdots\!:\!\alpha_n)$, with $\alpha_i \ne 0$.
We set $m_j = \lcm (a_i, a_j)$.
Then the polynomials in the set
\[
\Set{ \alpha_i^{m_j/a_i} x_j^{m_j/a_j} - \alpha_j^{m_j/a_j} x_i^{m_j/a_i} | j \in \{0, 1, \dots, n\} \setminus \{i\}}
\]
isolate $\msp$, and the first assertion follows by considering the divisors $(g_j = 0) \cap V$.

Suppose that $\msp \notin \cap_{j \ne k} (x_j = 0) \cap V$, where $k \ne i$.
Then the natural projection
\[
V \hookrightarrow \mbP (a_0, \dots, a_n) \ratmap \mbP (a_0, \dots, \hat{a}_k, \dots, a_n) =: \mbP'
\]
is a finite morphism in a neighborhood of $\msp \in V$.
The polynomials in the set
\[
\Set{ \alpha_i^{m_j/a_i} x_j^{m_j/a_j} - \alpha_j^{m_j/a_j} x_i^{m_j/a_i} | j \in \{0, 1, \dots, n\} \setminus \{i, k\}}
\]
isolate $\tau (\msp) \in \mbP'$, that is, the common zero locus of the polynomials in the above set is a finite set of points on $\mbP'$.
It follows that the common zero locus of the polynomials in the same set, considered as a subset of $\mbP (a_0, \dots, a_n)$, is again a finite set of points since $\tau$ is a finite morphism.
This shows that the above set isolates $\msp$ and the second assertion is proved.

Suppose that $c = 2$ and $\msp \notin \cap_{j \ne k_1, k_2} (x_j = 0) \cap V$, where $k_1, k_2 \ne i$.
Then the natural projection
\[
V \hookrightarrow \mbP (a_0, \dots, a_n) \ratmap \mbP (a_0, \dots, \hat{a}_{k_1}, \dots, \hat{a}_{k_2}, \dots, a_n) =: \mbP''
\]
is a finite morphism in a neighborhood of $\msp \in V$.
By the similar argument as above considering the set
\[
\Set{ \alpha_i^{m_j/a_i} x_j^{m_j/a_j} - \alpha_j^{m_j/a_j} x_i^{m_j/a_i} | j \in \{0, 1, \dots, n\} \setminus \{i, k_1, k_2\}},
\]
we obtain the third assertion.
\end{proof}

\begin{Lem} \label{lem:Pisoldiv1}
Let $X$ be a prime Fano $3$-fold weighted hypersurface of index $1$ which belongs to one of the families listed in Table~\ref{table:Fanohyp}.
Let $\msp \in X$ be a point contained in the smooth locus of the ambient weighted projective space.
Then the following assertions hold.
\begin{enumerate}
\item There exist $\msp$-isolating divisors of degree at most $4/(-K_X)^3$ except when the family number of $X$ is $25$ and $\msp \in (x = y = z = 0) \cap X$.
\item If $X$ belongs to a family which is given one of $\heartsuit, \diamondsuit, \clubsuit$ in the ``Case" column of Table~\ref{table:Fanohyp}, then there exist $\msp$-isolating divisors of degree at most $2/(-K_X^3)$.
\end{enumerate}
\end{Lem}

\begin{proof}
Let $X = X_d \subset \mbP (a_0, a_1, \dots, a_4) =: \mbP$ with $1 = a_0 \le a_1 \le \cdots \le a_4$, where $d = a_1 + \cdots + a_4$.
Let $x, y, z, t, w$ be homogeneous coordinates of weights $a_0, a_1, a_2, a_3, a_4$, respectively.

We prove (2).
We consider $X$ belonging to a family which is given one of $\heartsuit, \diamondsuit, \clubsuit$ in the ``Case" column of Table~\ref{table:Fanohyp}.
Then the assertions follow from the following observations.
\begin{itemize}
\item Case: $X$ belongs to a family with $\heartsuit$.
We have
\[
\max \{ \, \lcm (a_i, a_j) \mid i, j \in \{0, 1, 2, 3, 4\} \,\} \le 2/(-K_X)^3,
\]
and the assertion follows from Lemma~\ref{lem:isoldivs} (1-a).
\item Case: $X$ belongs to a family with $\diamondsuit$.
We have $a_3 < a_4$, $a_4 \mid d$, and
\[
\max \{\, \lcm (a_i, a_j) \mid i, j \in \{0, 1, 2, 3\} \,\} \le 2/(-K_X)^3.
\]
It follows from $a_3 < a_4$ and $a_4 \mid d$ that the monomial $w^{d/a_4}$ appears in the defining polynomial of $X$ with nonzero coefficient, and thus $(x = y = z = t = 0) \cap X = \emptyset$.
The assertion now follows from Lemma~\ref{lem:isoldivs} (1-b).
\item Case: $X$ belongs to a family with $\clubsuit$.
We have $a_2 < a_3 < a_4$, $a_3 \mid d$ and
\[
\max \{\, \lcm (a_i, a_j) \mid i, j \in \{0, 1, 2, 4\} \,\} \le 2/(-K_X)^3.
\]
It follows from $a_2 < a_3 < a_4$ and $a_3 \mid d$ that the monomial $t^{d/a_3}$ appears in the defining polynomial of $X$ with nonzero coefficient, and thus $(x = y = z = w = 0) \cap X = \emptyset$.
The assertion now follows from Lemma~\ref{lem:isoldivs} (1-b).
\end{itemize}

We prove (1).
It is enough to consider $X$ whose family number $\msi$ belongs to
\[
\{16, 17, 25, 31, 36, 38, 46, 47\}.
\]
\begin{itemize}
\item Case: $\msi \in \{17, 31, 38\}$.
We have
\[
\max \{\, \lcm (a_i, a_j) \mid i, j \in \{0, 1, 2, 3, 4\} \,\} \le 4/(-K_X)^3,
\]
and the assertion follows from Lemma~\ref{lem:isoldivs} (1-a).
\item Case: $\msi \in \{16, 36, 46, 47\}$.
In this case $a_3 \mid d$ and $w^{d/a_3}$ appears in the defining polynomial of $X$.
Hence $(x = y = z = w = 0) \cap X = \emptyset$.
We have
\[
\max \{\, \lcm (a_i, a_j) \mid i, j \in \{0, 1, 2, 4\} \,\} \le 4/(-K_X)^3,
\]
and the assertion follows from Lemma~\ref{lem:isoldivs} (1-b).
\item Case $\msi = 25$:
By the assumption, we have $\msp \in U_x \cup U_y \cup U_z$.
For $i = 0, 1, 2$, it is straightforward to check
\[
\max \{\, \lcm (a_i, a_j) \mid j \in \{0, 1, 2, 3, 4\} \,\} \le 4/(-K_X)^3.
\]
Thus the assertion follows from Lemma~\ref{lem:isoldivs} (2-a).
\end{itemize}
The proof is now completed.
\end{proof}

\begin{Lem} \label{lem:Pisoldiv2}
Let $X$ be a prime Fano $3$-fold weighted complete intersection of codimension $2$ and index $1$ which belongs to one of the families listed in Table~\ref{table:FanoWCI}.
Let $\msp \in X$ be a point which is contained in the smooth locus of the ambient weighted projective space.
Then there exist $\msp$-isolating divisors of degree at most $2/(-K_X)^3$.
\end{Lem}

\begin{proof}
Let $X = X_{d_1, d_2} \subset \mbP (a_0, \dots, a_5)$ with $a_0 \le \cdots \le a_5$ and $d_1 < d_2$.
Let $x, y, z, t, v, w$ be homogeneous coordinates of weights $a_0, a_1, a_2, a_3, a_4, a_5$, respectively, and let $f_1, f_2 \in \mbC [x, y, z, t, v, w]$ be the defining polynomials of $X$ of degree $d_1, d_2$, respectively.
We denote by $\msi$ the family number of $X$.

We consider the case $\msi = 60, 71, 76, 78, 84, 85$.
We have
\[
\max \{\, \lcm (a_i, a_j) \mid i, j \in \{0, 1, \dots, 5\} \,\} \le 2/(-K_X)^3,
\]
and thus the assertion follows from Lemma~\ref{lem:isoldivs} (1-a).
\item Case: $\msi = 8, 20, 31, 45, 47, 59$.
In this case $a_4 < a_5$, $a_5 \mid d_2$ and
\[
\max \{\, \lcm (a_i, a_j) \mid i, j \in \{0, 1, 2, 3, 4\} \,\} \le 2/(-K_X)^3.
\]
Then $w^{d_2/a_5} \in f_2$ since $a_4 < a_5$ and $a_5 \mid d_2$, and thus $(x = y = z = t = v = 0) \cap X = \emptyset$.
The assertion follows from Lemma~\ref{lem:isoldivs} (1-b).

We consider the case $\msi = 24, 37, 51, 64, 75$.
In this case $a_2 < a_3 < a_4 < a_5$, $d_1 < a_3 + a_5$, $a_3 \mid d_1$, $a_5 \mid d_2$, and
\[
\max \{\, \lcm (a_i, a_j) \mid i, j \in \{0, 1, 2, 4\} \,\} \le 2/(-K_X)^3.
\]
Then
\[
f_1 (0, 0, 0, t, 0, w) = \alpha t^{d_1/a_3}, \quad f_2 (0, 0, 0, t, 0, w) = \beta w^{d_2/a_4}
\]
for some nonzero $\alpha, \beta \in \mbC$.
It follows that $(x = y = z = v = 0) \cap X = \emptyset$, and thus the assertion follows from Lemma~\ref{lem:isoldivs} (1-c).

We consider the case $\msi = 14$.
In this case we have $(x = y  z = t = 0) \cap X = \emptyset$ and
\[
\max \{ \, \lcm (a_i, a_j) \mid i, j \in \{0, 1, 2, 3\} \,\} = 2 < 2/(-K_X)^3 = 4.
\]
Thus the assertion follows from Lemma~\ref{lem:isoldivs} (1-c), and the proof is completed.
\end{proof}

\begin{Rem}
In Lemmas~\ref{lem:Pisoldiv1} and~\ref{lem:Pisoldiv2}, we do not assume that $\msp$ is a smooth point of $X$, and the case where $\msp$ is a Gorenstein singular point of $X$ is covered.
This will be utilized in the proof of Lemma~\ref{lem:singPisoldiv} below.
\end{Rem}

\begin{Prop} \label{prop:WCIsmpt}
Let $X$ be as in Setting~\ref{setting:FanoWCI}.
Then no smooth point on $X$ is a maximal center.
\end{Prop}

\begin{proof}
Let $\msp \in X$ be a smooth point.
If $X$ is a weighted hypersurface and its family number is $25$, and $\msp$ is contained in $(x = y = z = 0) \cap X$, then we can repeat the argument of the proof of \cite[Lemma 2.1.5]{CP17} without any change and conclude that $\msp$ is not a maximal center.
In what follows, we assume that $\msp \notin (x = y = z = 0) \cap X$ when $X$ is a weighted hypersurface and its family number is $25$.

Suppose that $X$ admits a smooth point $\msp \in X$ which is a maximal center.
Then there exists a movable linear system $\mcM \subset \left| - n K_X \right|$ such that $\msp$ is a center of non-canonical singularities of $(X, \frac{1}{n} \mcM)$.
Let $D_1, D_2 \in \mcM$ be general members of $\mcM$.
By Lemmas~\ref{lem:Pisoldiv1} and~\ref{lem:Pisoldiv2}, we can take a divisor $T \in \left| - l K_X \right|$ for some positive integer $l \le 4/(-K_X^3)$ such that $\Supp T$ passes through $\msp$ and does not contain any component of $D_1 \cap D_2$.
Thus we have
\[
4 n^2 \ge l n^2 (-K_X^3) = (T \cdot D_1 \cdot D_2) > 4 n^2,
\]
where the last inequality follows from the $4 n^2$-inequality for smooth points (see Theorem~\ref{thm:4nineq}).
This is a contradiction.
\end{proof}

\subsection{Exclusion of $cA_1$ points}
\label{sec:WCIGorsing}

The aim of this section is to exclude singularities of type $cA_1$ on $X$ as a a maximal singularity for $X$ as in Setting~\ref{setting:FanoWCI}.

\begin{Lem} \label{lem:singPisoldiv}
Let $X$ be as in Setting~\ref{setting:FanoWCI} and let $\msp \in X$ be a singular point contained in the smooth locus of the ambient weighted projective space.
Then there exist $\msp$-isolating divisors of degree at most $2/(-K_X^3)$.
\end{Lem}

\begin{proof}
By Lemmas~\ref{lem:Pisoldiv1} and~\ref{lem:Pisoldiv2}, it remains to consider the case where $X$ is a weighted hypersurface whose family number $\msi$ belongs to
\[
\{16, 17, 25, 31, 36, 38, 46, 47\}.
\]

We consider the case $\msi = 38$.
In this case $X = X_{18} \subset \mbP (1, 2, 3, 5, 8)$.
We claim that the set $\Delta := (x = y = z = 0) \cap X$ is contained in  $\Qsm (X)$.
Let $f = f (x, y, z, t, w)$ be the defining polynomial of $X$.
We have $f = \alpha t^2 w + \beta y w^2 + \gamma z t^3 + G$, where $\alpha, \beta, \gamma \in \mbC$ and $G = G (x, y, z, t, w) \in (x, y, z)^2$.
We have $(\alpha, \gamma) \ne (0, 0)$ and $\beta \ne 0$.
If $\alpha \ne 0$, then $\Delta$ consists of $2$ points which are the $\frac{1}{5} (1, 2, 3)$ point and the $\frac{1}{8} (1, 3, 5)$ point, and hence $X$ is quasi-smooth along $\Delta$.
Suppose that $\alpha = 0$.
In this case $\Gamma = (x = y = z = 0)$ is a curve and it is straightforward to check that $X$ is quasi-smooth along $\Gamma$ since $\beta \ne 0$ and $\gamma \ne 0$.
The claim is proved.
It follows that $\msp \in U_x \cup U_y \cup U_z$ and there exist $\msp$-isolating divisors of degree at most $3 a_4$ by Lemma~\ref{lem:isoldivs} (2-a), and we have $3 a_4 = 24 \le 2/(-K_X)^3 = 80/3$.

Finally, we consider the case $\msi \in \{16, 17, 25, 31, 36, 46, 47\}$.
In this case $X = X_d \subset \mbP (1, 1, a_2, a_3, a_4)$ with $d = 1 + a_2 + a_3 + a_4$ and $a_2 \ge 2$.
We claim that $\msp \in U_x \cup U_y$.
Assume to the contrary that $\msp \in (x = y = 0) \cap X$.
Recall that $\msp$ is a singular point of $X$, hence it is either a hypersurface singular point or a cyclic quotient terminal singular point (see Remark~\ref{rem:Setting}).
In the latter case $\msp$ is contained in the singular locus of the ambient weighted projective space.
Hence $\msp \in X$ is a hypersurface singularity and in particular $X$ is not quasi-smooth at $\msp$.
Thus we have
\begin{equation} \label{eq:prtfxy0}
\frac{\prt f (0,0,z,t,w)}{\prt z} (\msp) = \frac{\prt f (0,0,z,t,w)}{\prt t} (\msp) = \frac{\prt f (0,0,z,t,w)}{\prt w} (\msp) = 0.
\end{equation}
An explicit description of $f (0, 0, z, t, w)$ is given in the second column of Table~\ref{table:descrf}, where the coefficients $\alpha, \beta, \dots$ satisfy the conditions given in the third column.
It is then straightforward to see that either the equations in \eqref{eq:prtfxy0} have only trivial solution or they imply that $\msp \in \{\msp_z, \msp_t, \msp_w\}$.
We have a contradiction in the former case, and the latter case is impossible since any point in $\{\msp_z, \msp_t, \msp_w\} \cap X$ is a quotient singular point of $X$.
Thus the claim is proved, and we have $\msp \in U_x \cup U_y$.
By Lemma~\ref{lem:isoldivs} (2-a), there exist $\msp$-isolating divisors of degree most $a_4$.
It is easy to check $a_4 \le 2/(-K_X)^3$ (see the fourth column of Table~\ref{table:descrf}) and the proof is completed.
\end{proof}

\begin{table}[h]
\renewcommand{\arraystretch}{1.1}
\begin{center}
\caption{Descriptions of $f (0, 0, z, t, w)$}
\label{table:descrf}
\begin{tabular}{cccc}
\hline
No. & $f (0, 0, z, t, w)$ & Conditions & $2/(-K_X)^3$  \\
\hline
16 & $\alpha w^2 z + \beta t^3 + \gamma t^2 z^2 + \delta t z^4 + \varepsilon z^6$ & $\alpha \ne 0, \beta \ne 0$ & $20/3$ \\
17 & $\alpha z^4 + \beta t^3 + \gamma t^2 w + \delta w^3$  & $\alpha \ne 0$  & $8$ \\
25 & $\alpha w t^2 + \beta t^3 z + \gamma z^5$ & $(\alpha, \beta) \ne (0,0), \gamma \ne 0$ & $56/5$ \\
31 & $\alpha w^2 z + \beta w t^2 + \gamma z^4$ & $\alpha \ne 0, \gamma \ne 0$ & $15$ \\
36 & $\alpha w^2 z + \beta t^3 + \gamma t z^3$ & $\alpha \ne 0, \beta \ne 0, \gamma \ne 0$ & $56/3$ \\
46 & $\alpha t^3 + \beta z^7$ & $\alpha \ne 0, \beta \ne 0$ & $20$ \\
47 & $\alpha w^2 z + \beta t^3$ & $\alpha \ne 0, \beta \ne 0$ & $80/3$
\end{tabular}
\end{center}
\end{table}

\begin{Prop} \label{prop:WCIcApt}
Let $X$ be as in Setting~\ref{setting:FanoWCI}, and let $\msp \in X$ be a singular point of type $cA_1$.
Then $\msp$ is not a maximal center.
\end{Prop}

\begin{proof}
Suppose that $\msp \in X$ is a maximal center.
Then there exists a movable linear system $\mcM \sim - n K_X$ such that $\msp$ is the center of a non-canonical singularities of the pair $(X, \frac{1}{n} \mcM)$.
Let $D_1, D_2 \in \mcM$ be general members.
Then by Lemma~\ref{lem:singPisoldiv}, we can take a divisor $T \in \left|- l K_X \right|$ for some positive integer $l \le 2/(-K_X)^3$ such that $\Supp T$ passes through $\msp$ and does not contain any component of $D_1 \cap D_2$.
Thus, we have
\[
2 n^2 \ge l n^2 (-K_X^3) = (T \cdot D_1 \cdot D_2) > 2 n^2,
\]
where the last inequality follows from Theorem~\ref{thm:ineqcA1}.
This is a contradiction.
\end{proof}

\subsection{Links centered at quotient points and proof of Theorem~\ref{thm:WCI}}

\begin{Prop} \label{prop:WCIqpt}
Let $X$ be as in Setting~\ref{setting:FanoWCI} and let $\msp \in X$ be a terminal quotient singular point.
If the point $\msp$ is a maximal center, then there exists an elementary self-link $X \ratmap X$ of type II initiated by the Kawamata blow-up at $\msp$.
\end{Prop}

\begin{proof}
Suppose that $X$ is a weighted hypersurface.
Then the same arguments as in \cite[Chapter 5]{CP17} prove the assertion.
To be more precise, although $X$ is assumed to be quasi-smooth in \cite{CP17}, the arguments in \cite[Chapter 5]{CP17} in fact relies only on the quasi-smoothness of $X$ along $X \cap \Sing (\mbP)$, where $\mbP$ is the ambient weighted projective $4$-space, besides the basic assumption that $X$ is a prime Fano $3$-fold.

Similarly, for weighted complete intersections of codimension 2, the same arguments as in \cite[Section 7]{OkadaI} and \cite[Section 4]{AZ} prove the assertion.
\end{proof}

\begin{Rem}
It should be mentioned that, for some weighted hypersurface $X$ of index $1$ which is not listed Table~\ref{table:Fanohyp}, the arguments in \cite[Chapter 5]{CP17} do rely on the entire quasi-smoothness of $X$.
For example, the singular point of type $\frac{1}{3} (1, 1, 2)$ on a quasi-smooth member $X_{17} \subset \mbP (1, 2, 3, 5, 7) =: \mbP$ of family No.~$33$ is excluded as a maximal center in \cite[Chapter 5]{CP17} and the proof of this fact does rely on the quasi-smoothness of $X$ (not only on the quasi-smoothness along $X \cap \Sing \mbP$).
\end{Rem}

\begin{proof}[Proof of Theorem~\ref{thm:WCI}]
This is a consequence of Lemma~\ref{lem:cribirrig} and Propositions~\ref{prop:WCIcurve}, \ref{prop:WCIsmpt}, \ref{prop:WCIcApt} and~\ref{prop:WCIqpt}.
\end{proof}

\section{Del Pezzo fibrations of degree $1$}

The aim of this section is to prove the following by the method of maximal singularities developed in \cite{PukdP} and \cite{OkadadP} combined with Corti's inequality for smooth points and $cA_1$ points.

\begin{Thm} \label{thm:dP}
Let $\pi \colon X \to \mbP^1$ be a del Pezzo fibration of degree $1$ with only $cA_1$ points.
If $X/\mbP^1$ satisfies the $K^2$-condition~\eqref{eq:K^2}, then it is birationally superrigid.
\end{Thm}

The rest of this section is devoted to the proof of Theorem~\ref{thm:dP}.
In what follows, let $\pi \colon X \to \mbP^1$ be a del Pezzo fibration of degree $1$ with only $cA_1$ singularities and assume that $X$ satisfies the $K^2$-condition~\eqref{eq:K^2}.
We denote by $F \in \Pic (X)$ the fiber class.

We recall that by \cite[Page 117]{PukdP} the variety $X$ can be embedded in a toric $\mbP (1, 1, 2, 3)$-bundle $P \to \mbP^1$ in such a way that the morphism $\pi$ coincides with the restriction of $P \to \mbP^1$.
A fiber of $\pi$ is a Gorenstein del Pezzo surface (which can be non-normal), and it is an irreducible and reduced hypersurface of degree $6$ in $\mbP (1, 1, 2, 3)$.

\begin{Prop}[{\cite[Section 3]{PukdP}}] \label{prop:dPcurve}
Under the above setting, no curve on $X$ is a center of non-canonical singularities of the pair $(X, \frac{1}{n} \mcM)$ for any movable linear system $\mcM \subset \left| - n K_X + m F \right|$ on $X$.
\end{Prop}

\begin{proof}
In \cite[Section 3]{PukdP} this statement is proven under the assumption that $X$ is smooth, but the same proof works in Gorenstein case.
\end{proof}

We say that an irreducible curve $C \subset X$ is {\it horizontal} if $\pi (C) = \mbP^1$ and that it is {\it vertical} if $\pi (C)$ is a point.
A $1$-cycle $C$ is {\it horizontal} (resp.\ {\it vertical}) if every irreducible curve in $\Supp C$ is horizontal (resp.\ vertical).

\begin{proof}[Proof of Theorem~\ref{thm:dP}]
Suppose that $X/\mbP^1$ is not birationally superrigid.
Then there exists a Mori fiber space $Y/T$ and a birational map $\sigma \colon X \ratmap Y$ which is not square.
Let $\mcM \subset \left| - n K_X + m F \right|$ be the proper transform on $X$ of a very ample complete linear system on $Y$.
The $K^2$-condition~\eqref{eq:K^2} implies that $-K_X$ is not in the movable cone of $X$, and thus $m \ge 0$.
By Proposition~\ref{prop:dPcurve} and \cite[Proposition 2.7]{OkadadP}, there are points $\msp_1, \dots, \msp_k \in X$ in distinct $\pi$-fibers and positive rational numbers $\lambda_1, \dots, \lambda_k$ such that $\msp_1, \dots, \msp_k$ are centers of non-canonical singularities of the pair $(X, \frac{1}{n} \mcM - \sum \lambda_i F_i)$ and $\sum \lambda_i > m/n$, where $F_i$ is the $\pi$-fiber containing $\msp_i$.

Let $D_1, D_2$  be general members  of $\mcM$ and let $Z := D_1 \cdot D_2$.
We decompose $Z$ into the vertical and the horizontal components:
$Z = Z^v + Z^h$, and write
\[
Z^v = \sum Z_i^v,
\]
where the support of $Z_i^v$ is in a $\pi$-fiber and $Z^h$ is horizontal.
By the Corti's inequality Theorem~\ref{thm:Cortiineq}, there are numbers $t_i$, $i = 1, \dots, k$, with~$0 < t_i \le 1$ such that
\[
\mult_{\msp_i} Z^h + t_i \mult_{\msp_i} Z^v_i > 2 n^2 + 4 \lambda_i t_i n^2.
\]
We have
\[
\mult_{\msp_i} Z^h \le (F_i \cdot Z^h) = (F \cdot Z) = (F \cdot D_1 \cdot D_2) = n^2,
\]
and thus
\[
\mult_{\msp_i} Z_i^v > \frac{1}{t_i} n^2 + 4 \lambda_i n^2 \ge n^2 + 4 \lambda_i n^2.
\]
Note that $F_i$ is isomorphic to an irreducible and reduced weighted hypersurface of degree $6$ in $\mbP (1, 1, 2, 3)$.
By \cite[Lemma 2.10]{OkadadP}, we can take a curve $C_i \in |\mcO_{F_i} (2)|$ which passes through $\msp_i$ and which does not contain any component of $\Supp (Z_i^v)$.
Then we can take a divisor $H_i \in \left|- 2 K_X + r_i F \right|$, where~$r_i$ is a sufficiently large integer, such that $H_i|_{F_i} = C_i$.
We have
\[
(-K_X \cdot Z^v_i) = \frac{1}{2} (H_i \cdot Z^v_i) \ge \frac{1}{2} \mult_{\msp_i} Z^v_i > \frac{1}{2} n^2 + 2 \lambda_i n^2,
\]
and, by the inequality $\sum \lambda_i > m/n$, we have
\begin{equation} \label{eq:dP-1}
(-K_X \cdot Z^v) = \sum_{i=1}^k (-K_X \cdot Z^v_i) > \frac{k}{2} n^2 + 2 n^2 \sum_{i=1}^k \lambda_i > \frac{1}{2} n^2 + 2 m n.
\end{equation}

We define $\ell := -K_X \cdot F \in \bNE (X)$ so that $\mbR_{\ge 0} \cdot \ell \subset \bNE (X)$ is the extremal ray corresponding to $\pi$.
Denote by $[Z^h] \in \bNE (X)$ the class of $Z^h$.
We can write $[Z^h] = \alpha (-K_X)^2 + \beta \ell$ for some numbers $\alpha, \beta$.
We have $\alpha = (F \cdot Z^h) = n^2$, and the $K^2$-condition~\eqref{eq:K^2} implies that $\beta \ge 0$.
It follows that
\[
(-K_X \cdot Z^h) = (-K_X \cdot \alpha (-K_X)^2 + \beta \ell) \ge n^2 (-K_X)^3.
\]
Combining this with
\[
(-K_X \cdot Z) = (-K_X) \cdot (- n K_X + m F)^2 = n^2 (-K_X)^3 + 2 m n,
\]
we obtain the inequality
\begin{equation} \label{eq:dP-2}
(-K_X \cdot Z^v) = (-K_X \cdot Z) - (-K_X \cdot Z^h) \le 2 m n.
\end{equation}
Two inequalities \eqref{eq:dP-1} and \eqref{eq:dP-2} are impossible, and this shows that $X/\mbP^1$ is birationally superrigid.
\end{proof}


\end{document}